\documentclass[a4paper,11pt]{article}
\usepackage{makeidx}
\usepackage{graphicx}
\usepackage[english,activeacute]{babel}
\usepackage[latin1]{inputenc}
\usepackage{amsmath}
\usepackage{amsfonts}
\usepackage{amssymb,latexsym}
\usepackage{colortbl}
\usepackage{multicol}
\usepackage{hyperref}
\usepackage{xcolor}
\usepackage{orcidlink}

\newtheorem{theorem}{Theorem}

\newtheorem{corollary}{Corollary}

\newtheorem{lemma}{Lemma}

\newtheorem{proposition}{Proposition}

\newenvironment{proof}[1][Proof]{\noindent\textbf{#1.} }{\ \rule{0.5em}{0.5em}}
\begin{document}

\title{Higher--order recurrence relations, Sobolev--type inner products and matrix factorizations}
\author{{Carlos Hermoso$^{1}$\orcidlink{0000-0002-5556-1839}}, {Edmundo J. Huertas$^{1,\dag}$\orcidlink{0000-0001-6802-3303}}, {Alberto Lastra$^{1}$\orcidlink{0000-0002-4012-6471}}, {Francisco Marcell\'{a}n$^{2}$\orcidlink{0000-0003-4331-4475}}\\
\\
$^{1}$Departamento de F\'{i}sica y Matem\'{a}ticas, Universidad de Alcal\'{a}\\
Ctra. Madrid-Barcelona, Km. 33,600\\
28805 - Alcal\'{a} de Henares, Madrid, Spain\\
carlos.hermoso@uah.es, edmundo.huertas@uah.es, alberto.lastra@uah.es\\
\\
[10pt] $^{2}$Departamento de Matem\'{a}ticas, Universidad Carlos III de Madrid\\
Avda. de la Universidad, 30\\
28911 - Legan\'{e}s, Madrid, Spain\\
pacomarc@ing.uc3m.es}
\date{\emph{\today}}
\maketitle

\begin{abstract}
It is well known that Sobolev-type orthogonal polynomials with respect to measures supported on the real line satisfy higher-order recurrence relations and these can be expressed as a (2N+1)-banded symmetric semi-infinite matrix. In this paper we state the connection between these (2N+1)-banded matrices and the Jacobi matrices associated with the three-term recurrence relation satisfied by the standard sequence of orthonormal polynomials with respect to the 2-iterated Christoffel transformation of the measure.


\end{abstract}


\section{Introduction}

\label{Sec1-Introduction}



Given a vector of measures $(\mu _{0},\mu _{1},\cdots ,\mu _{m})$ such that $%
\mu _{k}$ is supported on a set $E_{k}$, $k=0,1,\cdots ,m,$ of the real
line, let consider the Sobolev inner product

\begin{equation*}
\langle f, g\rangle_{S}=\sum _{k=0}^{m}\int_{E_{k}} f^{(k)}(x) g^{(k)}(x)
d\mu_{k}(x).
\end{equation*}

Several examples of sequences of orthogonal polynomials with respect to the
above inner products have been studied in the literature (see \cite{MX15})
as a recent survey).

\begin{enumerate}
\item When $E_{k}$, $k=0,1,\cdots ,m,$ are infinite subsets of the real
line. (Continuous Sobolev)

\item When $E_{0}$ is an infinite subset of the real line and $E_{k}$, $%
k=1,\cdots ,m,$ are finite subsets (Sobolev type)

\item When $E_{m}$ is an infinite subset of the real line and $E_{k}$, $%
k=0,\cdots ,m-1,$ are finite subsets.
\end{enumerate}

In the above cases, the three term recurrence relation that every sequence
of orthogonal polynomials with respect to a measure supported on an infinite
subset of the real line does not hold, This is a direct consequence of the
fact that the multiplication operator by $x$ is not symmetric with respect
to any of the above mentioned situations.

\smallskip

In the Sobolev type case, you get a multiplication operator by a polynomial
intimately related with the support of the discrete measures. In \cite{K90}
an illustrative example when $d\mu _{0}=x^{\alpha }e^{-x}dx$, $\alpha >-1$, $%
x\in \lbrack 0,+\infty )$ and $d\mu _{k}(x)=M_{k}\delta (x)$, $M_{k}\geq 0$,
$k=1,2,\cdots ,m,$ has been studied. In general, there exists a symmetric
multiplication operator for a general Sobolev inner product if and only if
the measures $\mu _{1},\cdots ,\mu _{m}$ are discrete (see \cite{ELMMR95}).
On the other hand, in \cite{D93} the study of the general inner product such
that the multiplication operator by a polynomial is a symmetric operator
with respect to the inner product has been done. The representation of such
inner products is given as well as the associated inner product. Assuming
some extra conditions, you get a Sobolev-type inner product. Notice that
there is an intimate relation about these facts and higher order recurrence
relations that the sequences of orthonormal polynomials with respect to the
above general inner products satisfy. A connection with matrix orthogonal
polynomials has been stated in \cite{DW95}.

\smallskip

When you deal with the Sobolev type inner product, a lot of contributions
have emphasized on the algebraic properties of the corresponding sequences
of orthogonal polynomials in terms of the polynomials orthogonal with
respect to the measure $\mu _{0}.$ The case $m=1$ has been studied in \cite%
{AMRR92}, where representation formulas for the new family as well as the
study of he distribution of their zeros have been analyzed. The particular
case of Laguerre Sobolev type orthogonal polynomials has been introduced and
deeply analyzed in \cite{KM93}. Outer ratio asymptotics when the measure
belongs to the Nevai class and some extensions to a more general framework
of Sobolev type inner products have been analyzed in \cite{MV93} and \cite%
{LMV95}. For measures supported on unbounded intervals, asymptotic
properties of Sobolev type orthogonal polynomials have been studied for
Laguerre measures (see \cite{HPMQ-subm13}, \cite{MaZeFeHu11}) and, in a more
general framework, in \cite{MM06}.

\smallskip

The aim of our contribution is to analyze the higher order recurrence
relation that a sequence of Sobolev type orthonormal polynomials satisfies
when you consider $d\mu _{0}=d\mu +M\delta (x-c)$ and $d\mu _{1}=N\delta
(x-c),$ where $M,N$ are nonnegative real numbers. In a first step, we obtain
connection formulas between such Sobolev type orthonormal polynomials and
the standard ones associated with the measures $d\mu $ and $(x-c)^{2}d\mu ,$
respectively. A matrix analysis of the five diagonal symmetric matrix
associated with such a higher order recurrence relation is presented taking
into account the $QR$ factorization of the shifted symmetric Jacobi
associated with the orthonormal polynomials with respect to the measure $%
d\mu $. The shifted Jacobi matrix associated with $(x-c)^{2}d\mu $ is $RQ$
(see \cite{BI-JCAM92}, \cite{KG83}). Our approach is quite different and it
is based on the iteration of the Cholesky factorization of the symmetric
Jacobi matrices associated with $d\mu $ and $(x-c)d\mu $, respectively (see
\cite{MBFM04}, \cite{Ga02}).

\smallskip

These polynomial perturbations of measures are known in the literature as
Christoffel perturbations (see \cite{Ga04} and \cite{Zhe97}). They
constitute examples of linear spectral transformations. The set of linear
spectral transformations is generated by Christoffel and Geronimus
transformations (see \cite{Zhe97}). The connection with matrix analysis
appears in \cite{DM-NA14} and \cite{DGM-LAA14} in terms of an inverse
problem for bilinear forms. On the other hand, Christoffel transformations
of the above type are related to Gaussian rules as it is studied in \cite%
{Ga02}. For a more general framework about perturbations of bilinear forms
and Hessenberg matrices as representations of a polynomial multiplication
operator in terms of sequences of orthonormal polynomials associated with
such bilinear forms, see \cite{BM-LAA06}.

\smallskip

The structure of the manuscript is as follows. Section 2 contains the basic
background about polynomial sequences orthogonal with respect to a measure
supported on an infinite set of the real line. We will call them standard
orthogonal polynomial sequences. In Section 3 we present several connection
formulas between the sequences of standard orthonormal polynomials
associated with the measures $d\mu $ and $(x-c)^{2}d\mu $ and the
orthonormal polynomials with respect to a Sobolev-type inner product. We
give alternative proofs to those presented in \cite{Thesis-H12}. In Section
4, we deduce the coefficients of the three term recurrence relation for the
orthonormal polynomials associated with the measure $(x-c)^{2}d\mu .$ In
Section 5 we study the five term recurrence relation that orthonormal
polynomials with respect to the Sobolev -type inner product satisfy. Section
6 deals with the connection between the shifted Jacobi matrices associated
with the measures $d\mu $ and $(x-c)^{2}d\mu $ in terms of $QR$
factorizations. In a next step, taking into account the Cholesky
factorization of the symmetric five diagonal matrix associated with the
multiplication operator $(x-c)^{2}$ in terms of the Sobolev-type orthonormal
polynomials by commuting the factors we get the square of the shifted Jacobi
matrix associated with the measure $(x-c)^{2}d\mu .$ Finally, in Section 7
we show an illustrative example in the framework of Laguerre- Sobolev type
inner products when $c=-1.$ Notice that in the literature, the authors have
focused the interest in the case $c=0$ and the analysis of the corresponding
differential operator such that the above polynomials are their
eigenfunctions (see \cite{KKB98}, \cite{CLE19} and \cite{CLE21}).




\section{Preliminaries}

\label{Sec2-Preliminaries}



Let $\mu $ be a finite and positive Borel measure supported on an infinite
subset $E$ of the real line such that all the integrals%
\begin{equation*}
\mu _{n}=\int_{E}x^{n}d\mu (x),
\end{equation*}%
exist for $n=0,\,1,\,2,\ldots $. $\mu _{n}$ is said to be the \textit{moment
of order }$n$\textit{\ of the measure }$\mu $. The measure $\mu $ is said to
be absolutely continuous with respect to the Lebesgue measure if there
exists a non-negative function $\omega (x)$ such that $d\mu (x)=\omega (x)dx$%
.

In the sequel, let $\mathbb{P}$ denote the linear space of polynomials in
one real variable with real coefficients, and let $\{P_{n}(x)\}_{n\geq 0}$
be the sequence of polynomials in $\mathbb{P}$ with leading coefficient
equal to one (monic OPS, or MOPS in short), orthogonal with respect to the
inner product $\langle \cdot ,\cdot \rangle _{\mu }:\mathbb{P}\times \mathbb{%
P}\rightarrow \mathbb{R}$ associated with $\mu $%
\begin{equation}
\langle f,g\rangle _{\mu }=\int_{E}f(x)g(x)d\mu (x).  \label{S1-InnProd-mu}
\end{equation}%
It induces the norm $||f||_{\mu }^{2}=\langle f,f\rangle _{\mu }$. Under
these considerations, these polynomials satisfy the following three term
recurrence relation%
\begin{equation}
xP_{n}(x)=P_{n+1}(x)+\beta _{n}P_{n}(x)+\gamma _{n}P_{n-1}(x),\qquad n\geq 0,
\label{S1-3TRRmonic}
\end{equation}%
where for every\ $n\geq 1$, $\gamma _{n}$ is a positive real number and $%
\beta _{n}$, $n\geq 0$\ is a real number.

The $n$--th reproducing kernel for $\omega (x)$ is%
\begin{equation}
K_{n}(x,y)=\sum_{k=0}^{n}\frac{P_{k}(x)P_{k}(y)}{||P_{k}||_{\mu }^{2}}%
,\qquad n\geq 0.  \label{S1-Knxy-P}
\end{equation}%
Because of the Christoffel-Darboux formula, see \cite{Chi78}, it may also be
expressed as%
\begin{equation}
K_{n}(x,y)=\frac{1}{||P_{n}||_{\mu }^{2}}\frac{%
P_{n+1}(x)P_{n}(y)-P_{n}(x)P_{n+1}(y)}{x-y},n\geq 0.  \label{S1-Knxy-CD-P}
\end{equation}%
The confluent formula becomes%
\begin{equation*}
K_{n}(x,x)=\sum_{k=0}^{n}\frac{[P_{k}(x)]^{2}}{||P_{k}||_{\mu }^{2}}=\frac{%
P_{n+1}^{\prime }(x)P_{n}(x)-P_{n}^{\prime }(x)P_{n+1}(x)}{||P_{n}||_{\mu
}^{2}},n\geq 0.
\end{equation*}%
We introduce the following usual notation for the partial derivatives of the
$n$-th reproducing kernel $K_{n}(x,y)$%
\begin{equation*}
\frac{\partial ^{j+k}K_{n}(x,y)}{\partial x^{j}\partial y^{k}}%
=K_{n}^{(j,k)}(x,y),\quad 0\leq j,k\leq n.
\end{equation*}%
We will use the expression of the first $y$--derivative of (\ref{S1-Knxy-P})
evaluated at $y=c$%
\begin{equation*}
K_{n}^{(0,1)}(x,c)=\frac{1}{||P_{n}||_{\mu }^{2}}\times
\end{equation*}%
\begin{equation}
\left[ \frac{P_{n+1}(x)P_{n}(c)-P_{n}(x)P_{n+1}(c)}{(x-c)^{2}}+\frac{%
P_{n+1}(x)P_{n}^{\prime }(c)-P_{n}(x)P_{n+1}^{\prime }(c)}{x-c}\right] ,
\label{S1-Knxy01der-P}
\end{equation}%
and the following confluent formulas%
\begin{equation}
K_{n}^{(0,1)}(c,c)=\frac{1}{||P_{n}||_{\mu }^{2}}\left[ \frac{%
P_{n}(c)P_{n+1}^{\prime \prime }(c)-P_{n+1}(c)P_{n}^{\prime \prime }(c)}{2}%
\right] ,n\geq 0,  \label{K01ccconfl}
\end{equation}%
\begin{equation*}
K_{n-1}^{(1,1)}(c,c)=\frac{1}{||P_{n}||_{\mu }^{2}}\times
\end{equation*}%
\begin{equation*}
\left[ \frac{P_{n}(c)P_{n+1}^{\prime \prime \prime
}(c)-P_{n+1}(c)P_{n}^{\prime \prime \prime }(c)}{6}+\frac{P_{n}^{\prime
}(c)P_{n+1}^{\prime \prime }(c)-P_{n+1}^{\prime }(c)P_{n}^{\prime \prime }(c)%
}{2}\right] ,n\geq 0,
\end{equation*}%
whose proof can be found in \cite[Sec. 2.1.2]{Thesis-H12}.

\smallskip

We will denote by $\{p_{n}(x)\}_{n\geq 0}$ the orthonormal polynomial
sequence with respect to the measure $\mu $. Obviously,

\begin{equation*}
p_{n}(x)=\frac{P_{n}(x)}{||P_{n}||_{\mu }}=r_{n}x^{n}+\text{\textit{lower
degree terms}.}
\end{equation*}%
Notice that%
\begin{equation*}
r_{n}=\frac{1}{||P_{n}||_{\mu }}.
\end{equation*}%
Using orthonormal polynomials, the Christoffel-Darboux formula (\ref%
{S1-Knxy-CD-P}) reads%
\begin{equation}
K_{n}(x,y)=\sum_{k=0}^{n}p_{k}(x)p_{k}(y)=\frac{r_{n}}{r_{n+1}}\frac{%
p_{n+1}(x)p_{n}(y)-p_{n}(x)p_{n+1}(y)}{x-y}  \label{S1-CD-orthonormals}
\end{equation}%
and its confluent form is%
\begin{equation*}
K_{n}(x,x)=\sum_{k=0}^{n}[p_{k}(x)]^{2}=\frac{r_{n}}{r_{n+1}}\left(
p_{n+1}^{\prime }(x)p_{n}(x)-p_{n}^{\prime }(x)p_{n+1}(x)\right) .
\end{equation*}

Next we define the Christoffel canonical transformation of a measure $\mu $
(see \cite{MBFM04}, \cite{Yoon02} and \cite{Zhe97}). Let $\mu $ be a
positive Borel measure supported on $E\subseteq \mathbb{R}$, and assume $%
c\notin E$. Here and in the sequel, $\{P_{n}^{[k]}(x)\}_{n\geq 0}$ will
denote the MOPS with respect to the inner product%
\begin{equation}
\langle f,g\rangle _{\lbrack k]}=\int_{E}f(x)g(x)d\mu ^{\lbrack k]},\quad
d\mu ^{\lbrack k]}=(x-c)^{k}d\mu ,\quad k\geq 0,\quad c\notin E.
\label{S1-k-iter-OP}
\end{equation}%
$\{P_{n}^{[k]}(x)\}_{n\geq 0}$ is said to be the $k$-iterated Christoffel
MOPS with respect to the above \textbf{standard} inner product. If $k=1$ we
have the \textbf{Christoffel canonical} perturbation of $\mu $. It is well
known that, in such a case, $P_{n}(c)\neq 0$, and (see \cite[(7.3)]{Chi78})%
\begin{equation*}
P_{n}^{[1]}(x)=\frac{1}{(x-c)}\left[ P_{n+1}(x)-\frac{P_{n+1}(c)}{P_{n}(c)}%
P_{n}(x)\right] =\frac{\Vert P_{n}\Vert _{\mu }^{2}}{P_{n}(c)}K_{n}(x,c),
\end{equation*}%
are the monic polynomials orthogonal with respect to the modified measure $%
d\mu ^{\lbrack 1]}$. They are known in the literature as monic \textit{%
kernel polynomials}. If $k>1,$ then we have the $k$-iterated Christoffel
transformation of $d\mu $. In the sequel, we will denote%
\begin{equation*}
||P_{n}^{[k]}||_{[k]}^{2}=\int_{E}[P_{n}^{[k]}(x)]^{2}(x-c)^{k}d\mu
\end{equation*}%
and $x_{n,r}^{[k]},$ $r=1,2,...,n,$ will denote the zeros of $P_{n}^{[k]}(x)$
arranged in an increasing order. Since $P_{n}^{[2]}(x)$ are the polynomials
orthogonal with respect to (\ref{S1-k-iter-OP}) when $k=2$ we have%
\begin{equation*}
(x-c)^{2}P_{n}^{[2]}(x)=\frac{%
\begin{vmatrix}
P_{n+2}(x) & P_{n+1}(x) & P_{n}(x) \\
P_{n+2}(c) & P_{n+1}(c) & P_{n}(c) \\
P_{n+2}^{\prime }(c) & P_{n+1}^{\prime }(c) & P_{n}^{\prime }(c)%
\end{vmatrix}%
}{%
\begin{vmatrix}
P_{n+1}(c) & P_{n}(c) \\
P_{n+1}^{\prime }(c) & P_{n}^{\prime }(c)%
\end{vmatrix}%
},
\end{equation*}%
i.e.,%
\begin{equation}
(x-c)^{2}P_{n}^{[2]}(x)=P_{n+2}(x)-d_{n}P_{n+1}(x)+e_{n}P_{n}(x),
\label{LagKer[2]monic}
\end{equation}%
where%
\begin{eqnarray}
d_{n} &=&\frac{P_{n+2}(c)P_{n}^{\prime }(c)-P_{n+2}^{\prime }(c)P_{n}(c)}{%
P_{n+1}(c)P_{n}^{\prime }(c)-P_{n+1}^{\prime }(c)P_{n}(c)},  \notag \\
e_{n} &=&\frac{P_{n+2}(c)P_{n+1}^{\prime }(c)-P_{n+2}^{\prime }(c)P_{n+1}(c)%
}{P_{n+1}(c)P_{n}^{\prime }(c)-P_{n+1}^{\prime }(c)P_{n}(c)}
\label{LagKer[2]moniccoeff} \\
&=&\frac{||P_{n+1}||_{\mu }^{2}}{||P_{n}||_{\mu }^{2}}\frac{K_{n+1}(c,c)}{%
K_{n}(c,c)}=\frac{r_{n}^{2}}{r_{n+1}^{2}}\frac{K_{n+1}(c,c)}{K_{n}(c,c)}>0.
\notag
\end{eqnarray}%
Similar determinantal formulas can be obtained for $k>2$. For orthonormal
polynomials the above expression reads%
\begin{equation*}
(x-c)^{2}\frac{p_{n}^{[2]}(x)}{r_{n}^{[2]}}=\frac{p_{n+2}(x)}{r_{n+2}}-d_{n}%
\frac{p_{n+1}(x)}{r_{n+1}}+e_{n}\frac{p_{n}(x)}{r_{n}}
\end{equation*}%
or, equivalently.
\begin{equation}
(x-c)^{2}p_{n}^{[2]}(x)=\frac{r_{n}^{[2]}}{r_{n+2}}p_{n+2}(x)-d_{n}\frac{%
r_{n}^{[2]}}{r_{n+1}}p_{n+1}(x)+e_{n}\frac{r_{n}^{[2]}}{r_{n}}p_{n}(x).
\label{confop2}
\end{equation}%
Furthermore, from \cite[Theorem 2.5]{Szego75} we conclude that%
\begin{equation*}
||P_{n}^{[2]}||_{[2]}^{2}=-\frac{P_{n+1}^{[1]}(c)}{P_{n}^{[1]}(c)}%
||P_{n}^{[1]}||_{[1]}^{2}=\frac{P_{n+1}^{[1]}(c)}{P_{n}^{[1]}(c)}\frac{%
P_{n+1}(c)}{P_{n}(c)}||P_{n}||_{\mu }^{2}\,.
\end{equation*}

On the other hand, taking (\ref{LagKer[2]monic}) into account%
\begin{equation*}
e_{n}=\frac{\int_{E}(x-c)^{2}P_{n}^{[2]}(x)P_{n}(x)d\mu }{%
\int_{E}P_{n}^{2}(x)d\mu }=\frac{||P_{n}^{[2]}||_{[2]}^{2}}{||P_{n}||_{\mu
}^{2}}=\frac{||P_{n+1}||_{\mu }^{2}}{||P_{n}||_{\mu }^{2}}\frac{K_{n+1}(c,c)%
}{K_{n}(c,c)}
\end{equation*}%
which implies that%
\begin{equation}
r_{n}^{[2]}=r_{n+1}\left( \frac{K_{n}(c,c)}{K_{n+1}(c,c)}\right) ^{1/2}.
\label{r2rnKK}
\end{equation}%
Replacing in (\ref{confop2}), the orthonormal version of the connection
formula (\ref{LagKer[2]monic}) reads%
\begin{equation*}
(x-c)^{2}p_{n}^{[2]}(x)=
\end{equation*}%
\begin{eqnarray*}
&&\left( \frac{K_{n}(c,c)}{K_{n+1}(c,c)}\right) ^{1/2}\times \left( \frac{%
r_{n+1}}{r_{n+2}}p_{n+2}(x)-d_{n}p_{n+1}(x)+e_{n}\frac{r_{n+1}}{r_{n}}%
p_{n}(x)\right) \\
&=&\left( \frac{K_{n}(c,c)}{K_{n+1}(c,c)}\right) ^{1/2}\times \left( \frac{%
||P_{n+2}||_{\mu }}{||P_{n+1}||_{\mu }}p_{n+2}(x)-d_{n}p_{n+1}(x)+e_{n}\frac{%
||P_{n}||_{\mu }}{||P_{n+1}||_{\mu }}p_{n}(x)\right) .
\end{eqnarray*}

\smallskip

In this contribution we will focus our attention on following inner product (%
\textit{Sobolev type inner product}%
\begin{equation}
\langle f,g\rangle _{S}=\int_{E}f(x)g(x)d\mu +Mf(c)g(c)+Nf^{\prime
}(c)g^{\prime }(c),\quad f,\,g\in \mathbb{P},  \label{S1-DicrSob-Inn}
\end{equation}%
where $\mu $ is a positive Borel measure supported on $E=[a,b]\subseteq
\mathbb{R}$, $c\notin E,$ and $M$, $N\geq 0$. In general, $E$ can be a
bounded or unbounded interval of the real line. Let $\{S_{n}^{M,N}(x)\}_{n%
\geq 0}$ denote the monic orthogonal polynomial sequence (MOPS in short)
with respect to (\ref{S1-DicrSob-Inn}). These polynomials are known in the
literature as \textit{Sobolev-type} or \textit{discrete Sobolev} orthogonal
polynomials. It is worth to point out that many properties of the standard
orthogonal polynomials are lost when an inner product as (\ref%
{S1-DicrSob-Inn}) is considered.

\smallskip



\section{The 3TRR for the 2-iterated orthogonal polynomials}

\label{Sec3-3TRR4P2}



In order to obtain the corresponding symmetric Jacobi matrix, in this
section will find the coefficients of the three term recurrence relation
satisfied by the $2-$iterated orthonormal polynomials $\{p_{n}^{[2]}(x)\}_{n%
\geq 0}$. First, we deal with the monic orthogonal polynomials $%
\{P_{n}^{[2]}(x)\}_{n\geq 0}.$Taking into account it is a standard sequence
we will have%
\begin{equation*}
x\,P_{n}^{[2]}(x)=P_{n+1}^{[2]}(x)+\kappa _{n}P_{n}^{[2]}(x)+\tau
_{n}P_{n-1}^{[2]}(x),\quad n\geq 0,
\end{equation*}%
where%
\begin{equation*}
\kappa _{n}=\frac{\langle x\,P_{n}^{[2]}(x),P_{n}^{[2]}(x)\rangle _{\lbrack
2]}}{\langle P_{n}^{[2]}(x),P_{n}^{[2]}(x)\rangle _{\lbrack 2]}},\qquad \tau
_{n}=\frac{\langle x\,P_{n}^{[2]}(x),P_{n-1}^{[2]}(x)\rangle _{\lbrack 2]}}{%
\langle P_{n-1}^{[2]}(x),P_{n-1}^{[2]}(x)\rangle _{\lbrack 2]}}.
\end{equation*}%
In order to obtain the explicit expression of the above coefficients, we
first study the numerator in $\kappa _{n}$. Taking into account (\ref%
{S1-k-iter-OP}) and (\ref{LagKer[2]monic}) we have%
\begin{eqnarray}
\langle x\,P_{n}^{[2]}(x),P_{n}^{[2]}(x)\rangle _{\lbrack 2]} &=&\langle
x\,P_{n}^{[2]}(x),(x-c)^{2}P_{n}^{[2]}(x)\rangle  \notag \\
&=&\langle x\,P_{n}^{[2]}(x),P_{n+2}(x)\rangle -d_{n}\langle
x\,P_{n}^{[2]}(x),P_{n+1}(x)\rangle +e_{n}\langle
P_{n}^{[2]}(x),xP_{n}(x)\rangle  \notag \\
&=&-d_{n}||P_{n+1}||_{\mu }^{2}+e_{n}\langle P_{n}^{[2]}(x),xP_{n}(x)\rangle
.  \notag
\end{eqnarray}%
Next, applying (\ref{S1-3TRRmonic})%
\begin{eqnarray*}
\langle P_{n}^{[2]}(x),xP_{n}(x)\rangle &=&\langle
P_{n}^{[2]}(x),P_{n+1}(x)\rangle +\beta _{n}\langle
P_{n}^{[2]}(x),P_{n}(x)\rangle +\gamma _{n}\langle
P_{n}^{[2]}(x),P_{n-1}(x)\rangle \\
&=&\beta _{n}||P_{n}||_{\mu }^{2}+\gamma _{n}\langle
P_{n}^{[2]}(x),P_{n-1}(x)\rangle .
\end{eqnarray*}%
Taking into account (\ref{LagKer[2]monic})
\begin{equation*}
P_{n-1}(x)=\frac{1}{e_{n-1}}(x-c)^{2}P_{n-1}^{[2]}(x)-\frac{1}{e_{n-1}}%
P_{n+1}(x)+\frac{d_{n-1}}{e_{n-1}}P_{n}(x)
\end{equation*}%
we obtain%
\begin{eqnarray*}
\langle P_{n}^{[2]}(x),P_{n-1}(x)\rangle &=&\langle P_{n}^{[2]}(x),\frac{1}{%
e_{n-1}}(x-c)^{2}P_{n-1}^{[2]}(x)-\frac{1}{e_{n-1}}P_{n+1}(x)+\frac{d_{n-1}}{%
e_{n-1}}P_{n}(x)\rangle \\
&=&\frac{1}{e_{n-1}}\langle P_{n}^{[2]}(x),P_{n-1}^{[2]}(x)\rangle _{\lbrack
2]}-\frac{1}{e_{n-1}}\langle P_{n}^{[2]}(x),P_{n+1}(x)\rangle \\
&&+\frac{d_{n-1}}{e_{n-1}}\langle P_{n}^{[2]}(x),P_{n}(x)\rangle \\
&=&\frac{d_{n-1}}{e_{n-1}}||P_{n}||_{\mu }^{2}.
\end{eqnarray*}%
Thus%
\begin{equation*}
\langle x\,P_{n}^{[2]}(x),P_{n}^{[2]}(x)\rangle _{\lbrack 2]}=\left( \beta
_{n}+\gamma _{n}\frac{d_{n-1}}{e_{n-1}}\right) e_{n}||P_{n}||_{\mu
}^{2}-d_{n}||P_{n+1}||_{\mu }^{2}.
\end{equation*}%
Next, we study the denominator in the expression of $\kappa _{n}$. From (\ref%
{LagKer[2]monic}) we have%
\begin{eqnarray*}
\langle P_{n}^{[2]}(x),P_{n}^{[2]}(x)\rangle _{\lbrack 2]} &=&\langle
P_{n}^{[2]}(x),(x-c)^{2}P_{n}^{[2]}(x)\rangle \\
&=&\langle P_{n}^{[2]}(x),P_{n+2}(x)\rangle -d_{n}\langle
P_{n}^{[2]}(x),P_{n+1}(x)\rangle \\
&&+e_{n}\langle P_{n}^{[2]}(x),P_{n}(x)\rangle \\
&=&e_{n}||P_{n}||_{\mu }^{2}.
\end{eqnarray*}%
Hence%
\begin{eqnarray*}
\kappa _{n} &=&\frac{\left( \beta _{n}+\gamma _{n}\frac{d_{n-1}}{e_{n-1}}%
\right) e_{n}||P_{n}||_{\mu }^{2}-d_{n}||P_{n+1}||_{\mu }^{2}}{%
||P_{n}^{[2]}||_{[2]}^{2}} \\
&=&\left( \beta _{n}+\gamma _{n}\frac{d_{n-1}}{e_{n-1}}\right) e_{n}\left(
\frac{r_{n}^{[2]}}{r_{n}}\right) ^{2}-d_{n}\left( \frac{r_{n}^{[2]}}{r_{n+1}}%
\right) ^{2}, \\
\tau _{n} &=&e_{n}\frac{||P_{n}||_{\mu }^{2}}{||P_{n-1}^{[2]}||_{[2]}^{2}}%
=\left( \frac{r_{n-1}^{[2]}}{r_{n}}\right) ^{2}e_{n}>0,
\end{eqnarray*}%
where%
\begin{eqnarray*}
d_{n} &=&\frac{r_{n+1}}{r_{n+2}}\frac{p_{n+2}(c)}{p_{n+1}(c)}+\frac{r_{n}}{%
r_{n+1}}\frac{p_{n}(c)}{p_{n+1}(c)}\frac{K_{n+1}(c,c)}{K_{n}(c,c)}, \\
e_{n} &=&\frac{\Vert P_{n+1}\Vert _{\mu }^{2}}{\Vert P_{n}\Vert _{\mu }^{2}}%
\frac{K_{n+1}(c,c)}{K_{n}(c,c)}=\left( \frac{r_{n}}{r_{n+1}}\right) ^{2}%
\frac{K_{n+1}(c,c)}{K_{n}(c,c)}>0.
\end{eqnarray*}%
Hence, we have proved the following



\begin{proposition}
\label{PROP01}The monic sequence $\{P_{n}^{[2]}(x)\}_{n\geq 0}$ satisfies
the three term recurrence relation%
\begin{equation*}
x\,P_{n}^{[2]}(x)=P_{n+1}^{[2]}(x)+\kappa _{n}P_{n}^{[2]}(x)+\tau
_{n}P_{n-1}^{[2]}(x),\quad n\geq 0,
\end{equation*}%
with $P_{-1}^{[2]}(x)=0$, $P_{0}^{[2]}(x)=1$, and%
\begin{eqnarray*}
\kappa _{n} &=&\left( \beta _{n}+\gamma _{n}\frac{d_{n-1}}{e_{n-1}}\right)
e_{n}\left( \frac{r_{n}^{[2]}}{r_{n}}\right) ^{2}-d_{n}\left( \frac{%
r_{n}^{[2]}}{r_{n+1}}\right) ^{2}, \\
\tau _{n} &=&\left( \frac{r_{n-1}^{[2]}}{r_{n+1}}\right) ^{2}\frac{%
K_{n+1}(c,c)}{K_{n}(c,c)}>0,
\end{eqnarray*}%
where, taking into account the explicit expressions for $d_{n}$ and $e_{n}$
given in (\ref{LagKer[2]monic}), we also have%
\begin{eqnarray*}
d_{n} &=&\frac{r_{n+1}}{r_{n+2}}\frac{p_{n+2}(c)}{p_{n+1}(c)}+\frac{r_{n}}{%
r_{n+1}}\frac{p_{n}(c)}{p_{n+1}(c)}\frac{K_{n+1}(c,c)}{K_{n}(c,c)}, \\
e_{n} &=&\left( \frac{r_{n}}{r_{n+1}}\right) ^{2}\frac{K_{n+1}(c,c)}{%
K_{n}(c,c)}>0.
\end{eqnarray*}
\end{proposition}



Observe that the orthonormal version of the above proposition is%
\begin{equation*}
p_{n+1}^{[2]}(x)=\left( x-\kappa _{n}\right) \frac{p_{n}^{[2]}(x)}{%
r_{n}^{[2]}/r_{n+1}^{[2]}}-\tau _{n}\frac{p_{n-1}^{[2]}(x)}{%
r_{n-1}^{[2]}/r_{n+1}^{[2]}},\quad n\geq 0,
\end{equation*}%
and, according to \cite[Th. 1.29, p.12-13]{Ga04},%
\begin{equation*}
p_{n+1}^{[2]}(x)=\left( x-\kappa _{n}\right) \frac{p_{n}^{[2]}(x)}{\sqrt{%
\tau _{n+1}}}-\tau _{n}\frac{p_{n-1}^{[2]}(x)}{\sqrt{\tau _{n+1}\tau _{n}}}%
,\quad n\geq 0,
\end{equation*}%
so we can conclude that
\begin{equation*}
\sqrt{\tau _{n+1}}=\frac{r_{n}^{[2]}}{r_{n+1}^{[2]}},\quad \sqrt{\tau
_{n+1}\tau _{n}}=\frac{r_{n-1}^{[2]}}{r_{n+1}^{[2]}}=\frac{r_{n}^{[2]}}{%
r_{n+1}^{[2]}}\frac{r_{n-1}^{[2]}}{r_{n}^{[2]}}.
\end{equation*}%
Therefore%
\begin{equation*}
\tau _{n}=\left( \frac{r_{n-1}^{[2]}}{r_{n+1}}\right) ^{2}\frac{K_{n+1}(c,c)%
}{K_{n}(c,c)}=\left( \frac{r_{n-1}^{[2]}}{r_{n}^{[2]}}\right) ^{2}.
\end{equation*}%
As a consequence,%
\begin{eqnarray}
\frac{K_{n+1}(c,c)}{K_{n}(c,c)} &=&\left( \frac{r_{n-1}^{[2]}}{r_{n}^{[2]}}%
\right) ^{2}\left( \frac{r_{n+1}}{r_{n-1}^{[2]}}\right) ^{2}=\left( \frac{%
r_{n+1}}{r_{n}^{[2]}}\right) ^{2},  \notag \\
e_{n} &=&\left( \frac{r_{n}}{r_{n}^{[2]}}\right) ^{2}>0,  \label{endnSimples}
\\
d_{n} &=&\frac{r_{n+1}}{r_{n+2}}\frac{p_{n+2}(c)}{p_{n+1}(c)}+\left( \frac{%
r_{n}}{r_{n}^{[2]}}\right) ^{2}\frac{r_{n+1}}{r_{n}}\frac{p_{n}(c)}{%
p_{n+1}(c)}.  \notag
\end{eqnarray}%
Replacing in $\kappa _{n}$ these alternative expressions for $e_{n}$\ and $%
d_{n}$\ we have%
\begin{equation*}
\kappa _{n}=\left( \beta _{n}+\gamma _{n}\frac{d_{n-1}}{e_{n-1}}\right)
e_{n}\left( \frac{r_{n}^{[2]}}{r_{n}}\right) ^{2}-d_{n}\left( \frac{%
r_{n}^{[2]}}{r_{n+1}}\right) ^{2},
\end{equation*}%
\begin{equation*}
\frac{d_{n-1}}{e_{n-1}}=\frac{\frac{r_{n}}{r_{n+1}}\frac{p_{n+1}(c)}{p_{n}(c)%
}+\left( \frac{r_{n-1}}{r_{n-1}^{[2]}}\right) ^{2}\frac{r_{n}}{r_{n-1}}\frac{%
p_{n-1}(c)}{p_{n}(c)}}{\left( \frac{r_{n-1}}{r_{n-1}^{[2]}}\right) ^{2}}%
=\left( \frac{r_{n-1}^{[2]}}{r_{n-1}}\frac{r_{n-1}^{[2]}}{r_{n-1}}\frac{r_{n}%
}{r_{n+1}}\frac{p_{n+1}(c)}{p_{n}(c)}+\frac{r_{n}}{r_{n-1}}\frac{p_{n-1}(c)}{%
p_{n}(c)}\right).
\end{equation*}%
Therefore%
\begin{eqnarray*}
\kappa _{n} &=&\beta _{n}+\gamma _{n}\left( \frac{r_{n-1}^{[2]}}{r_{n-1}}%
\frac{r_{n-1}^{[2]}}{r_{n-1}}\frac{r_{n}}{r_{n+1}}\frac{p_{n+1}(c)}{p_{n}(c)}%
+\frac{r_{n}}{r_{n-1}}\frac{p_{n-1}(c)}{p_{n}(c)}\right) \\
&&\qquad \qquad \qquad -\left( \frac{r_{n}^{[2]}}{r_{n+1}}\frac{r_{n}^{[2]}}{%
r_{n+2}}\frac{p_{n+2}(c)}{p_{n+1}(c)}+\frac{r_{n}}{r_{n+1}}\frac{p_{n}(c)}{%
p_{n+1}(c)}\right) .
\end{eqnarray*}

We have then proved the following



\begin{corollary}
\label{COROL01}The orthonormal polynomial sequence $\{p_{n}^{[2]}(x)\}_{n%
\geq 0}$ satisfies the three term recurrence relation%
\begin{equation}
\sqrt{\tau _{n+1}}p_{n+1}^{[2]}(x)=\left( x-\kappa _{n}\right)
p_{n}^{[2]}(x)-\sqrt{\tau _{n}}p_{n-1}^{[2]}(x),\quad n\geq 0,
\label{3TRRk2orthonormal}
\end{equation}%
with $p_{-1}^{[2]}(x)=0$, $p_{0}^{[2]}(x)=1/\sqrt{\tau _{0}}$, where%
\begin{eqnarray*}
\kappa _{n} &=&\left( \beta _{n}+\gamma _{n}\frac{d_{n-1}}{e_{n-1}}\right)
e_{n}\left( \frac{r_{n}^{[2]}}{r_{n}}\right) ^{2}-d_{n}\left( \frac{%
r_{n}^{[2]}}{r_{n+1}}\right) ^{2} \\
&=&\beta _{n}+\gamma _{n}\left( \left( \frac{r_{n-1}^{[2]}}{r_{n-1}}\right)
^{2}\frac{r_{n}}{r_{n+1}}\frac{p_{n+1}(c)}{p_{n}(c)}+\frac{r_{n}}{r_{n-1}}%
\frac{p_{n-1}(c)}{p_{n}(c)}\right) \\
&&\qquad \qquad \qquad -\left( \frac{r_{n}^{[2]}}{r_{n+1}}\frac{r_{n}^{[2]}}{%
r_{n+2}}\frac{p_{n+2}(c)}{p_{n+1}(c)}+\frac{r_{n}}{r_{n+1}}\frac{p_{n}(c)}{%
p_{n+1}(c)}\right) , \\
\tau _{n} &=&\left( \frac{r_{n-1}^{[2]}}{r_{n}^{[2]}}\right) ^{2}>0,
\end{eqnarray*}%
where $e_{n}$, and $d_{n}$ are given in (\ref{endnSimples}).
\end{corollary}



\section{Connection formulas}

\label{Sec4-ConnForm}



As we have seen in previous section, the connection formulas are the main
tool to study the analytical properties of new families of OPS, in terms of
other families of OPS with well-known analytical properties. Indeed, the
problem of finding such expressions is called \textit{the connection problem}%
, and it is of great importance in this context.

In this Section we present some results of \cite{Thesis-H12}, and which will
be useful later. We will give some alternative proofs of them. From now on,
let us denote by $\{s_{n}^{M,N}\}_{n\geq 0}$, $\{p_{n}\}_{n\geq 0}$ the
sequences of polynomials orthonormal with respect to (\ref{S1-DicrSob-Inn})
and (\ref{S1-InnProd-mu}), respectively. We will write%
\begin{eqnarray*}
s_{n}^{M,N}(x) &=&t_{n}x^{n}+\text{\textit{\ lower degree terms},\quad }t_{n}>0,
\\
p_{n}(x) &=&r_{n}x^{n}+\text{\textit{\ lower degree terms},\quad }r_{n}>0, \\
p_{n}^{[k]} (x) &=&r_{n}^{[k]}x^{n}+\text{\textit{\ lower degree terms},\quad }%
r_{n}^{[k]}>0.
\end{eqnarray*}

In the sequel the following notation will be useful. For every $k\in \mathbb{%
N}_{0}$, let us define $\mathbf{J}_{[k]}$ as the semi-infinite symmetric
Jacobi matrix associated with the measure $(x-c)^{k}d\mu $, verifying
\begin{equation*}
x\,\mathbf{\bar{p}}^{[k]}=\mathbf{J}_{[k]}\,\mathbf{\bar{p}}^{[k]},
\end{equation*}%
where $\mathbf{\bar{p}}^{[k]}$ stands for the semi-infinite column vector
with \textit{orthonormal} polynomial entries $\mathbf{\bar{p}}%
^{[k]}=[p_{0}^{[k]}(x),p_{1}^{[k]}(x),p_{2}^{[k]}(x),\ldots ]^{\intercal }$,
being $\{p^{[k]}(x)\}_{n\geq 0}$ the orthonormal polynomial sequence with
respect to the measure $(x-c)^{k}d\mu $ (\ref{S1-k-iter-OP}) . One has $%
\mathbf{\bar{p}}^{[0]}=\mathbf{\bar{p}}=[p_{0}(x),p_{1}(x),p_{2}(x),\ldots
]^{\intercal }$ being $\{p(x)\}_{n\geq 0} $ the orthonormal polynomial
sequence with respect to the standard measure $\mu $, and $\mathbf{J}_{[0]}=%
\mathbf{J}$ is the corresponding Jacobi matrix.

Next, we will present an expansion of the monic polynomials $S_{n}^{M,N}(x)$
in terms of polynomials $P_{n}(x)$ orthogonal with respect to $\mu $. When
necessary, we refer the reader to \cite[Th. 5.1]{Thesis-H12} for alternative
proofs to those presented here.



\begin{lemma}
\label{S2-LEMMA-01}%
\begin{equation}
S_{n}^{M,N}(x)=P_{n}(x)-M\,S_{n}^{M,N}(c)K_{n-1}(x,c)-N\,[S_{n}^{M,N}]^{%
\prime }(c)K_{n-1}^{(0,1)}(x,c)  \label{[Sec2]-ConnFormS}
\end{equation}%
where%
\begin{eqnarray}
S_{n}^{M,N}(c) &=&\frac{%
\begin{vmatrix}
P_{n}(c) & NK_{n-1}^{(0,1)}(c,c) \\
\lbrack P_{n}]^{\prime }(c) & 1+NK_{n-1}^{(1,1)}(c,c)%
\end{vmatrix}%
}{%
\begin{vmatrix}
1+MK_{n-1}(c,c) & NK_{n-1}^{(0,1)}(c,c) \\
MK_{n-1}^{(1,0)}(c,c) & 1+NK_{n-1}^{(1,1)}(c,c)%
\end{vmatrix}%
},  \label{[Sec2]-Sn-Det} \\
\lbrack S_{n}^{M,N}]^{\prime }(c) &=&\frac{%
\begin{vmatrix}
1+MK_{n-1}(c,c) & P_{n}(c) \\
MK_{n-1}^{(1,0)}(c,c) & [P_{n}]^{\prime }(c)%
\end{vmatrix}%
}{%
\begin{vmatrix}
1+MK_{n-1}(c,c) & NK_{n-1}^{(0,1)}(c,c) \\
MK_{n-1}^{(1,0)}(c,c) & 1+NK_{n-1}^{(1,1)}(c,c)%
\end{vmatrix}%
}.  \label{[Sec2]-Sprimn-Det}
\end{eqnarray}
\end{lemma}



\begin{proof}
We search for the expansion%
\begin{equation*}
S_{n}^{M,N}(x)=P_{n}(x)+\sum_{j=0}^{n-1}\varrho _{n,j}P_{j}(x),
\end{equation*}%
where%
\begin{equation*}
\varrho _{n,j}=\frac{\int_{E}S_{n}^{M,N}(x)P_{j}(x)d\mu }{||P_{j}||_{\mu
}^{2}}=-\frac{M\,S_{n}^{M,N}(c)P_{j}(c)}{||P_{j}||_{\mu }^{2}}-\frac{%
N\,[S_{n}^{M,N}]^{\prime }(c)[P_{n}]^{\prime }(c)}{||P_{j}||_{\mu }^{2}}.
\end{equation*}%
From these coefficients (\ref{[Sec2]-ConnFormS} follows). Next, having its
first derivative with respect to $x$, and taking $x=c$ we get%
\begin{eqnarray*}
P_{n}(c) &=&\left( 1+M\,K_{n-1}(c,c)\right)
S_{n}^{M,N}(c)+N\,K_{n-1}^{(0,1)}(c,c)[S_{n}^{M,N}]^{\prime }(c), \\
\lbrack P_{n}]^{\prime }(c) &=&M\,K_{n-1}^{(0,1)}(c,c)S_{n}^{M,N}(c)+\left(
1+N\,K_{n-1}^{(1,1)}(c,c)\right) [S_{n}^{M,N}]^{\prime }(c).
\end{eqnarray*}%
Solving the above linear system for $S_{n}^{M,N}(c)$ and $%
[S_{n}^{M,N}]^{\prime }(c)$ we obtain (\ref{[Sec2]-Sn-Det}) and (\ref%
{[Sec2]-Sn-Det}).

This completes the proof.
\end{proof}



From the above lemma, we can also express $S_{n}^{M,N}(x)$ as follows%
\begin{equation*}
S_{n}^{M,N}(x)=\frac{%
\begin{vmatrix}
P_{n}(x) & MK_{n-1}(x,c) & NK_{n-1}^{(0,1)}(x,c) \\
P_{n}(c) & 1+M\,K_{n-1}(c,c) & N\,K_{n-1}^{(0,1)}(c,c) \\
\lbrack P_{n}]^{\prime }(c) & M\,K_{n-1}^{(0,1)}(c,c) & 1+N%
\,K_{n-1}^{(1,1)}(c,c)%
\end{vmatrix}%
}{%
\begin{vmatrix}
1+M\,K_{n-1}(c,c) & N\,K_{n-1}^{(0,1)}(c,c) \\
M\,K_{n-1}^{(0,1)}(c,c) & 1+N\,K_{n-1}^{(1,1)}(c,c)%
\end{vmatrix}%
}.
\end{equation*}

In terms of the orthonormal polynomials (\ref{[Sec2]-ConnFormS}) becomes%
\begin{equation}
s_{n}^{M,N}(x)=\frac{t_{n}}{r_{n}}p_{n}(x)-M\,s_{n}^{M,N}(c)K_{n-1}(x,c)-N%
\,[s_{n}^{M,N}]^{\prime }(c)K_{n-1}^{(0,1)}(x,c).  \label{formula2}
\end{equation}


As a direct consequence of Lemma \ref{S2-LEMMA-01}, we get the following
result concerning the norm of the Sobolev type polynomials $S_{n}^{M,N}$



\begin{lemma}
\label{S2-LEMMA-02}For $c\in \mathbb{R}_{+}$ the norm of the monic Sobolev
type polynomials $S_{n}^{M,N}$, orthogonal with respect to (\ref%
{S1-DicrSob-Inn}) is%
\begin{equation*}
\frac{1}{t_{n}^{2}}=||S_{n}^{M,N}||_{S}^{2}=||P_{n}||_{\mu
}^{2}+M\,S_{n}^{M,N}(c)P_{n}(c)+N\,[S_{n}^{M,N}]^{\prime }(c)[P_{n}]^{\prime
}(c).
\end{equation*}
\end{lemma}



\begin{proof}
From (\ref{[Sec2]-ConnFormS}) we have%
\begin{equation*}
S_{n}^{M,N}(x)=P_{n}(x)-M\,S_{n}^{M,N}(c)K_{n-1}(x,c)-N\,[S_{n}^{M,N}]^{%
\prime }(c)K_{n-1}^{(0,1)}(x,c)
\end{equation*}%
and according to (\ref{S1-DicrSob-Inn}) we get%
\begin{equation*}
\langle S_{n}^{M,N}(x),S_{n}^{M,N}(x)\rangle _{S}=\langle
S_{n}^{M,N}(x),P_{n}(x)\rangle
+M\,S_{n}^{M,N}(c)P_{n}(c)+N\,[S_{n}^{M,N}]^{\prime }(c)[P_{n}]^{\prime }(c).
\end{equation*}%
This completes the proof.
\end{proof}



Next, we represent the Sobolev type orthogonal polynomials in terms of the
polynomial kernels associated with the sequence of orthonormal polynomials $%
\{p_{n}(x)\}_{n\geq 0}$ and its derivatives. Another proof of this result
can be found in \cite[Prop. 5.6, p. 115]{Thesis-H12}.



\begin{lemma}
\label{S2-LEMMA-03}The sequence of Sobolev type orthonormal polynomials $%
\{s_{n}^{M,N}(x)\}_{n\geq 0}$ can be expressed as%
\begin{eqnarray}
s_{n}^{M,N}(x) &=&\alpha _{n+1,n}p_{n+1}(x)+\alpha _{n,n}p_{n}(x)  \notag \\
&&-M\,s_{n}^{M,N}(c)K_{n+1}(x,c)-N\,[s_{n}^{M,N}]^{\prime
}(c)K_{n+1}^{(0,1)}(x,c),  \label{[Sec2]-resultado-1}
\end{eqnarray}%
where%
\begin{eqnarray*}
\alpha _{n+1,n} &=&M\,s_{n}^{M,N}(c)p_{n+1}(c)+N\,[s_{n}^{M,N}]^{\prime
}(c)[p_{n+1}]^{\prime }(c), \\
\alpha _{n,n} &=&\frac{t_{n}}{r_{n}}+M\,s_{n}^{M,N}(c)p_{n}(c)+N%
\,[s_{n}^{M,N}]^{\prime }(c)[p_{n}]^{\prime }(c).
\end{eqnarray*}
\end{lemma}



\begin{proof}
From (\ref{S1-CD-orthonormals})%
\begin{eqnarray*}
K_{n-1}(x,c) &=&K_{n+1}(x,c)-p_{n+1}(x)p_{n+1}(c)-p_{n}(x)p_{n}(c), \\
K_{n-1}^{(0,1)}(x,c) &=&K_{n+1}^{(0,1)}(x,c)-p_{n+1}(x)[p_{n+1}]^{\prime
}(c)-p_{n}(x)[p_{n}]^{\prime }(c).
\end{eqnarray*}%
Replacing in (\ref{formula2}) yields%
\begin{eqnarray*}
s_{n}^{M,N}(x) &=&\frac{t_{n}}{r_{n}}p_{n}(x)-M\,s_{n}^{M,N}(c)\left(
K_{n+1}(x,c)-p_{n+1}(x)p_{n+1}(c)-p_{n}(x)p_{n}(c)\right) \\
&&-N\,[s_{n}^{M,N}]^{\prime }(c)\left(
K_{n+1}^{(0,1)}(x,c)-p_{n+1}(x)[p_{n+1}]^{\prime
}(c)-p_{n}(x)[p_{n}]^{\prime }(c)\right) \\
&=&\left[ M\,s_{n}^{M,N}(c)p_{n+1}(c)+N\,[s_{n}^{M,N}]^{\prime
}(c)[p_{n+1}]^{\prime }(c)\right] p_{n+1}(x) \\
&&+\left[ \frac{t_{n}}{r_{n}}+M\,s_{n}^{M,N}(c)p_{n}(c)+N\,[s_{n}^{M,N}]^{%
\prime }(c)[p_{n}]^{\prime }(c)\right] p_{n}(x) \\
&&-M\,s_{n}^{M,N}(c)K_{n+1}(x,c)-N\,[s_{n}^{M,N}]^{\prime
}(c)K_{n+1}^{(0,1)}(x,c)
\end{eqnarray*}%
This completes the proof.
\end{proof}



Next we expand the polynomials $\{p_{n}(x)\}_{n\geq 0}$ in terms of the
polynomials $\{p_{n}^{[2]}(x)\}_{n\geq 0}$. This result is already addressed
in \cite[Prop. 5.7, p.116]{Thesis-H12} as well as in \cite{Ardila2022} but
we include here an alternative proof



\begin{lemma}
\label{S2-LEMMA-04}The sequence of polynomials $\{p_{n}(x)\}_{n\geq 0}$,
orthonormal with respect to $d\mu $, is expressed in terms of the $2-$%
iterated orthonormal polynomials\ $\{p_{n}^{[2]}(x)\}_{n\geq 0}$ as follows%
\begin{equation*}
p_{n}(x)=\xi _{n,n}p_{n}^{[2]}(x)+\xi _{n-1,n}p_{n-1}^{[2]}(x)+\xi
_{n-2,n}p_{n-2}^{[2]}(x),
\end{equation*}%
where
\begin{eqnarray*}
\xi _{n,n} &=&\frac{r_{n}}{r_{n}^{[2]}}=\frac{r_{n}}{r_{n+1}}\left( \frac{%
K_{n+1}(c,c)}{K_{n}(c,c)}\right) ^{1/2}=e_{n}^{1/2}, \\
\xi _{n-1,n} &=&-d_{n-1}\left( \frac{K_{n-1}(c,c)}{K_{n}(c,c)}\right) ^{1/2},
\\
\xi _{n-2,n} &=&\frac{r_{n-1}}{r_{n}}\left( \frac{K_{n-2}(c,c)}{K_{n-1}(c,c)}%
\right) ^{1/2}=\frac{r_{n-1}^{2}}{r_{n}r_{n+1}}e_{n-2}^{1/2}\,.
\end{eqnarray*}
\end{lemma}



\begin{proof}
Taking into account (\ref{LagKer[2]monic}), (\ref{r2rnKK}) and (\ref{confop2}%
), for the first coefficient we immediately have%
\begin{eqnarray*}
\xi _{n,n} &=&\langle p_{n}(x),p_{n}^{[2]}(x)\rangle _{\lbrack 2]}=\langle
p_{n}(x),(x-c)^{2}p_{n}^{[2]}(x)\rangle \\
&=&\frac{r_{n}}{r_{n}^{[2]}}=\frac{r_{n}}{r_{n+1}}\left( \frac{K_{n+1}(c,c)}{%
K_{n}(c,c)}\right) ^{1/2}=e_{n}^{1/2}.
\end{eqnarray*}%
For the second coefficient, from (\ref{confop2}) we have%
\begin{eqnarray*}
\xi _{n-1,n} &=&\langle p_{n}(x),p_{n-1}^{[2]}(x)\rangle _{\lbrack
2]}=\langle p_{n}(x),(x-c)^{2}p_{n-1}^{[2]}(x)\rangle \\
&=&\langle p_{n}(x),-d_{n-1}\left( \frac{K_{n-1}(c,c)}{K_{n}(c,c)}\right)
^{1/2}p_{n}(x)\rangle =-d_{n-1}\left( \frac{K_{n-1}(c,c)}{K_{n}(c,c)}\right)
^{1/2}.
\end{eqnarray*}%
Finally, for the last coefficient, we get%
\begin{eqnarray*}
\xi _{n-2,n} &=&\langle p_{n}(x),p_{n-2}^{[2]}(x)\rangle _{\lbrack
2]}=\langle p_{n}(x),(x-c)^{2}p_{n-2}^{[2]}(x)\rangle \\
&=&\frac{r_{n-2}^{[2]}}{r_{n}}=\frac{r_{n-1}}{r_{n}}\left( \frac{K_{n-2}(c,c)%
}{K_{n-1}(c,c)}\right) ^{1/2}=\frac{r_{n-1}^{2}}{r_{n}r_{n+1}}e_{n-2}^{1/2}.
\end{eqnarray*}%
This completes the proof.
\end{proof}



Next, let us obtain a third representation for the Sobolev type OPS in terms
of the polynomials orthonormal with respect to $(x-c)^{2}d\mu $. This
expression will be very useful to find the connection of these polynomials
with the matrix orthogonal polynomials, and we include the proof for the
convenience of the reader.



\begin{theorem}
\label{S2-THEOR-01} Let $\{s_{n}^{M,N}(x)\}_{n\geq 0}$ be the sequence
Sobolev-type polynomials orthonormal with respect to (\ref{S1-DicrSob-Inn}),
and let $\{p_{n}^{[2]}(x)\}_{n\geq 0}$\ be the sequence of polynomials
orthonormal with respect to the inner product (\ref{S1-k-iter-OP}) with $k=2$%
. Then, the following expression holds%
\begin{equation}
s_{n}^{M,N}(x)=\gamma _{n,n}p_{n}^{[2]}(x)+\gamma
_{n-1,n}p_{n-1}^{[2]}(x)+\gamma _{n-2,n}p_{n-2}^{[2]}(x),
\label{[Sec2]-ThirdCF-01}
\end{equation}%
where,%
\begin{equation*}
\gamma _{n,n}=\frac{t_{n}}{r_{n}^{[2]}}=\frac{t_{n}}{r_{n+1}}\left( \frac{%
K_{n+1}(c,c)}{K_{n}(c,c)}\right) ^{1/2},
\end{equation*}%
\begin{equation*}
\gamma _{n-1,n}=-\left( \frac{K_{n-1}(c,c)}{K_{n}(c,c)}\right) ^{1/2}
\end{equation*}%
\begin{equation*}
\times \left( d_{n-1}\frac{t_{n}}{r_{n}}+e_{n-1}\frac{r_{n}}{r_{n-1}}\left[
M\,s_{n}^{M,N}(c)p_{n-1}(c)+N\,[s_{n}^{M,N}]^{\prime }(c)[p_{n}]^{\prime }(c)%
\right] \right) ,
\end{equation*}%
\begin{equation*}
\gamma _{n-2,n}=\frac{r_{n-1}}{t_{n}}\left( \frac{K_{n-2}(c,c)}{K_{n-1}(c,c)}%
\right) ^{1/2}.
\end{equation*}
\end{theorem}



\begin{proof}
For $\gamma _{n,n}$, matching the leading coefficients of $s_{n}^{M,N}(x)$
and $p_{n}^{[2]}(x)$, it is a straightforward consequence to see that%
\begin{equation*}
\gamma _{n,n}=\langle s_{n}^{M,N}(x),p_{n}^{[2]}(x)\rangle _{\lbrack 2]}=%
\frac{t_{n}}{r_{n}^{[2]}}.
\end{equation*}%
Next from (\ref{r2rnKK})
\begin{equation*}
\gamma _{n,n}=\frac{t_{n}}{r_{n}^{[2]}}=\frac{t_{n}}{r_{n+1}}\left( \frac{%
K_{n+1}(c,c)}{K_{n}(c,c)}\right) ^{1/2}.
\end{equation*}%
For $\gamma _{n-1,n}$ we need some extra work. From (\ref{[Sec2]-resultado-1}%
) we have%
\begin{eqnarray*}
\gamma _{n-1,n} &=&\langle s_{n}^{M,N}(x),p_{n-1}^{[2]}(x)\rangle _{\lbrack
2]}=\int_{E}s_{n}^{M,N}(x)(x-c)^{2}p_{n-1}^{[2]}(x)d\mu \\
&=&\int_{E}s_{n}^{M,N}(x)\left[ -d_{n-1}\frac{r_{n-1}^{[2]}}{r_{n}}%
p_{n}(x)+e_{n-1}\frac{r_{n-1}^{[2]}}{r_{n-1}}p_{n-1}(x)\right] d\mu \\
&=&-d_{n-1}\frac{t_{n}}{r_{n}}\left( \frac{K_{n-1}(c,c)}{K_{n}(c,c)}\right)
^{1/2}+e_{n-1}\frac{r_{n}}{r_{n-1}}\left( \frac{K_{n-1}(c,c)}{K_{n}(c,c)}%
\right) ^{1/2}\int_{E}s_{n}^{M,N}(x)p_{n-1}(x)d\mu .
\end{eqnarray*}%
The last integral can be computed using (\ref{formula2})%
\begin{equation*}
\int_{E}s_{n}^{M,N}(x)p_{n-1}(x)d\mu =
\end{equation*}%
\begin{eqnarray*}
&&\int_{E}\left( -M\,s_{n}^{M,N}(c)K_{n-1}(x,c)-N\,[s_{n}^{M,N}]^{\prime
}(c)K_{n-1}^{(0,1)}(x,c)\right) p_{n-1}(x)d\mu \\
&=&-M\,s_{n}^{M,N}(c)p_{n-1}(c)-N\,[s_{n}^{M,N}]^{\prime }(c)[p_{n}]^{\prime
}(c).
\end{eqnarray*}%
Thus%
\begin{equation*}
\gamma _{n-1,n}=-\left( \frac{K_{n-1}(c,c)}{K_{n}(c,c)}\right) ^{1/2}
\end{equation*}%
\begin{equation*}
\times \left( d_{n-1}\frac{t_{n}}{r_{n}}+e_{n-1}\frac{r_{n}}{r_{n-1}}\left[
M\,s_{n}^{M,N}(c)p_{n-1}(c)+N\,[s_{n}^{M,N}]^{\prime }(c)[p_{n}]^{\prime }(c)%
\right] \right) .
\end{equation*}%
Finally, for the last coefficient we have
\begin{eqnarray*}
\gamma _{n-2,n} &=&\langle s_{n}^{M,N}(x),p_{n-2}^{[2]}(x)\rangle _{\lbrack
2]}=\langle s_{n}^{M,N}(x),(x-c)^{2}p_{n-2}^{[2]}(x)\rangle _{S} \\
&=&t_{n}r_{n-2}^{[2]}\langle S_{n}^{M,N}(x),(x-c)^{2}P_{n-2}^{[2]}(x)\rangle
_{S} \\
&=&t_{n}r_{n-2}^{[2]}||S_{n}^{M,N}||_{S}^{2}=\frac{r_{n-2}^{[2]}}{t_{n}}=%
\frac{r_{n-1}}{t_{n}}\left( \frac{K_{n-2}(c,c)}{K_{n-1}(c,c)}\right) ^{1/2}.
\end{eqnarray*}%
This completes the proof.
\end{proof}



\section{The five term recurrence relation}

\label{Sec5-5TRR}



In this section, we will obtain the five term recurrence relation that the
sequence of Sobolev-type orthonormal polynomials $\{s_{n}^{M,N}(x)\}_{n\geq
0}$ satisfies. We use orthonormal polynomials because all the matrices
associated with the multiplication operators we are dealing with are
symmetric. Later on, we will derive an interesting relation between the five
diagonal matrix $\mathbf{H}$ associated with the multiplication operator by $%
(x-c)^{2}$ in terms of the orthonormal basis $\{s_{n}^{M,N}(x)\}_{n\geq 0}$,
and the tridiagonal Jacobi matrix $\mathbf{J}_{[2]}$ associated with the
three term recurrence relation satisfied by the $2-$iterated orthonormal
polynomials $\{p_{n}^{[2]}(x)\}_{n\geq 0}$.

To do that, we will use the following remarkable fact



\begin{theorem}
\label{S2-THEOR-02} The multiplication operator by $(x-c)^{2}$\ is a
symmetric operator with respect to the discrete Sobolev inner product (\ref%
{S1-DicrSob-Inn}). In other words, for any $p(x),\,q(x)\in \mathbb{P}$, it
satisfies%
\begin{equation}
\langle (x-c)^{2}p(x),q(x)\rangle _{S}=\langle p(x),(x-c)^{2}q(x)\rangle
_{S}.  \label{[Sec3]-SymmInn-bis}
\end{equation}
\end{theorem}



\begin{proof}
The proof is a straightforward consequence of (\ref{S1-DicrSob-Inn}).
\end{proof}



Next, we will obtain the coefficients of the aforementioned five term
recurrence relation. Let consider the Fourier expansion of $%
(x-c)^{2}s_{n}^{M,N}(x)$ in terms of $\{s_{n}^{M,N}(x)\}_{n\geq 0}$%
\begin{equation}
(x-c)^{2}s_{n}^{M,N}(x)=\sum_{k=0}^{n+2}\rho _{k,n}s_{k}^{M,N}(x),
\label{[Sec3]-5TRR-Expn}
\end{equation}%
where%
\begin{equation*}
\rho _{k,n}=\left\langle (x-c)^{2}s_{n}^{M,N}(x),s_{k}^{M,N}(x)\right\rangle
_{S},\quad k=0,\ldots ,n+2.
\end{equation*}%
From (\ref{[Sec3]-SymmInn-bis})%
\begin{equation*}
\rho _{k,n}=\left\langle s_{n}^{M,N}(x),(x-c)^{2}s_{k}^{M,N}(x)\right\rangle
_{S},\quad k=0,\ldots ,n+2.
\end{equation*}%
Hence, $\rho _{k,n}=0$ for $k=0,\ldots ,n-3$. Taking into account that%
\begin{equation*}
\lbrack (x-c)^{2}s_{n}^{M,N}(x)]|_{x=c}=[(x-c)^{2}s_{n}^{M,N}(x)]^{\prime
}|_{x=c}=0,
\end{equation*}%
and using \cite[Th. 1, p. 174]{MaZeFeHu11} we get%
\begin{equation}
\langle (x-c)^{2}s_{n}^{M,N}(x),s_{k}^{M,N}(x)\rangle _{S}=\langle
s_{n}^{M,N}(x),s_{k}^{M,N}(x)\rangle _{\lbrack 2]}.  \label{[Sec3]-Property1}
\end{equation}%
Notice that%
\begin{equation}
\langle (x-c)^{2}s_{n}^{M,N}(x),s_{k}^{M,N}(x)\rangle _{S}=\langle
(x-c)^{2}s_{n}^{M,N}(x),s_{k}^{M,N}(x)\rangle .  \label{[Sec3]-Property2}
\end{equation}%
Next, using the connection formula (\ref{[Sec2]-ThirdCF-01}) we have%
\begin{eqnarray*}
\rho _{n+2,n} &=&\langle (x-c)^{2}s_{n}^{M,N}(x),s_{n+2}^{M,N}(x)\rangle
_{S}=\langle s_{n}^{M,N}(x),s_{n+2}^{M,N}(x)\rangle _{\lbrack 2]} \\
&=&\gamma _{n,n}\gamma _{n,n+2}=\frac{t_{n}}{t_{n+2}},
\end{eqnarray*}%
\begin{eqnarray*}
\rho _{n+1,n} &=&\langle (x-c)^{2}s_{n}^{M,N}(x),s_{n+1}^{M,N}(x)\rangle
_{S}=\langle s_{n}^{M,N}(x),s_{n+1}^{M,N}(x)\rangle _{\lbrack 2]} \\
&=&\gamma _{n,n}\gamma _{n,n+1}\,\langle
p_{n}^{[2]}(x),p_{n}^{[2]}(x)\rangle _{\lbrack 2]}+\gamma _{n-1,n}\gamma
_{n-1,n+1}\,\langle p_{n-1}^{[2]}(x),p_{n-1}^{[2]}(x)\rangle _{\lbrack 2]} \\
&=&\gamma _{n,n}\gamma _{n,n+1}+\gamma _{n-1,n}\gamma _{n-1,n+1},
\end{eqnarray*}%
\begin{eqnarray*}
\rho _{n,n} &=&\langle (x-c)^{2}s_{n}^{M,N}(x),s_{n}^{M,N}(x)\rangle
_{S}=\langle s_{n}^{M,N}(x),s_{n}^{M,N}(x)\rangle _{\lbrack 2]} \\
&=&\gamma _{n,n}^{2}\,\langle p_{n}^{[2]}(x),p_{n}^{[2]}(x)\rangle _{\lbrack
2]}+\gamma _{n-1,n}^{2}\,\langle p_{n-1}^{[2]}(x),p_{n-1}^{[2]}(x)\rangle
_{\lbrack 2]}+\gamma _{n-2,n}^{2}\,\langle
p_{n-2}^{[2]}(x),p_{n-2}^{[2]}(x)\rangle _{\lbrack 2]} \\
&=&\gamma _{n,n}^{2}+\gamma _{n-1,n}^{2}+\gamma _{n-2,n}^{2},
\end{eqnarray*}

\begin{eqnarray*}
\rho _{n-1,n} &=&\langle (x-c)^{2}s_{n}^{M,N}(x),s_{n-1}^{M,N}(x)\rangle
_{S}=\langle s_{n}^{M,N}(x),s_{n-1}^{M,N}(x)\rangle _{\lbrack 2]} \\
&=&\gamma _{n-1,n}\gamma _{n-1,n-1}\,\langle
p_{n-1}^{[2]}(x),p_{n-1}^{[2]}(x)\rangle _{\lbrack 2]}+\gamma _{n-2,n}\gamma
_{n-2,n-1}\,\langle p_{n-2}^{[2]}(x),p_{n-2}^{[2]}(x)\rangle _{\lbrack 2]} \\
&=&\gamma _{n-1,n-1}\gamma _{n-1,n}+\gamma _{n-2,n}\gamma _{n-2,n-1}=\rho
_{n,n-1}\,,
\end{eqnarray*}

\begin{eqnarray*}
\rho _{n-2,n} &=&\langle (x-c)^{2}s_{n}^{M,N}(x),s_{n-2}^{M,N}(x)\rangle
_{S}=\langle s_{n}^{M,N}(x),s_{n-2}^{M,N}(x)\rangle _{\lbrack 2]} \\
&=&\gamma _{n-2,n-2}\gamma _{n-2,n}\,\langle
p_{n-2}^{[2]}(x),p_{n-2}^{[2]}(x)\rangle _{\lbrack 2]} \\
&=&\gamma _{n-2,n-2}\gamma _{n-2,n}=\frac{t_{n-2}}{t_{n}}.
\end{eqnarray*}

Introducing the following notation%
\begin{equation*}
\rho _{n-2,n}=a_{n},\quad \rho _{n-1,n}=b_{n},\quad \rho _{n,n}=c_{n}\,,
\end{equation*}%
(\ref{[Sec3]-5TRR-Expn}) reads%
\begin{equation*}
(x-c)^{2}s_{n}^{M,N}(x)=
\end{equation*}%
\begin{equation}
a_{n+2}s_{n+2}^{M,N}(x)+b_{n+1}s_{n+1}^{M,N}(x)+c_{n}s_{n}^{M,N}(x)+b_{n}s_{n-1}^{M,N}(x)+a_{n}s_{n-2}^{M,N}(x),\quad n\geq 0,
\label{[Sec3]-Coeff-5TRR-2}
\end{equation}%
where, by convention,%
\begin{equation*}
s_{-2}^{M,N}(x)=s_{-1}^{M,N}(x)=0.
\end{equation*}



\section{A matrix approach}

\label{Sec6-MatrixAppr}



In this Section we will deduce an interesting relation between the five
diagonal matrix $\mathbf{H}$ associated with the multiplication operator by $%
(x-c)^{2}$ associated with the orthonormal Sobolev type orthonormal
polynomials and the Jacobi matrix $\mathbf{J}_{[2]}$ associated with the $2-$
iterated orthonormal polynomials $\{p_{n}^{[2]}(x)\}_{n\geq 0}$.

First, we deal with the matrix representation of (\ref{[Sec3]-Coeff-5TRR-2})%
\begin{equation}
(x-c)^{2}\mathbf{\bar{s}}^{M,N}=\mathbf{H\,\bar{s}}^{M,N},
\label{FormMatrizH}
\end{equation}%
where $\mathbf{H}$ is the five diagonal semi-infinite symmetric matrix%
\begin{equation}
\mathbf{H=}%
\begin{bmatrix}
c_{0} & b_{1} & a_{2} & 0 & \cdots \\
b_{1} & c_{1} & b_{2} & a_{3} & \cdots \\
a_{2} & b_{2} & c_{2} & b_{3} & \ddots \\
0 & a_{3} & b_{3} & c_{3} & \ddots \\
\vdots & \vdots & \ddots & \ddots & \ddots%
\end{bmatrix}%
,  \label{MatrixH}
\end{equation}%
and $\mathbf{\bar{s}}^{M,N}=[s_{0}^{M,N}(x),s_{1}^{M,N}(x),s_{2}^{M,N}(x),%
\ldots ]^{\intercal }$.

Next, from (\ref{[Sec2]-ThirdCF-01}) we get%
\begin{equation}
\mathbf{\bar{s}}^{M,N}=\mathbf{T\,\bar{p}}^{[2]}  \label{Prop-Sobpert2}
\end{equation}%
where $\mathbf{T}$\ is the lower triangular, semi-infinite, and nonsingular
matrix with positive diagonal entries%
\begin{equation}
\mathbf{T}=%
\begin{bmatrix}
\gamma _{0,0} &  &  &  &  \\
\gamma _{0,1} & \gamma _{1,1} &  &  &  \\
\gamma _{0,2} & \gamma _{1,2} & \gamma _{2,2} &  &  \\
\gamma _{0,3} & \gamma _{1,3} & \gamma _{2,3} & \gamma _{3,3} &  \\
&  & \ddots & \ddots & \ddots%
\end{bmatrix}
\label{MatrixT}
\end{equation}%
and $\mathbf{\bar{p}}^{[2]}=[p_{0}^{[2]}(x),p_{1}^{[2]}(x),p_{2}^{[2]}(x),%
\ldots ]^{\intercal }$. We will denote by $\mathbf{J}$ the Jacobi matrix
associated with the orthonormal sequence $\{p_{n}(x)\}_{n\geq 0}$, with
respect to the measure $d\mu $. As a consequence, we have%
\begin{equation*}
x\,\mathbf{\bar{p}}=\mathbf{J}\,\mathbf{\bar{p}}.
\end{equation*}%
Let $\mathbf{J}_{[2]}$ be the Jacobi matrix associated with the $2-$
iterated OPS $\{p_{n}^{[2]}(x)\}_{n\geq 0}$ . Notice that from
\begin{equation*}
x\,\mathbf{\bar{p}}^{[2]}=\mathbf{J}_{[2]}\,\mathbf{\bar{p}}^{[2]},
\end{equation*}%
we get%
\begin{equation}
\left( x-c\right) ^{2}\mathbf{\bar{p}}^{[2]}=\left( \mathbf{J}_{[2]}-c%
\mathbf{I}\right) ^{2}\mathbf{\bar{p}}^{[2]}.  \label{Prop-04}
\end{equation}

\smallskip

Starting with $(\mathbf{J}-c\mathbf{I})$, and assuming $c$ is located in the
left hand side of supp$(\mu ) $, all their leading principal submatrices are
positive definite, so we get the following Cholesky factorization%
\begin{equation}
\mathbf{J}-c\mathbf{I=LL}^{\intercal }.  \label{Chol00}
\end{equation}

Here $\mathbf{L}$ is a lower bidiagonal matrix with positive diagonal
entries. From \cite{BM-LAA06}\ we know%
\begin{equation}
\mathbf{J}_{[1]}-c\mathbf{I=\mathbf{L}^{\intercal }\mathbf{L}=\mathbf{L}_{1}%
\mathbf{L}_{1}^{\intercal }},  \label{Chol01}
\end{equation}

where $\mathbf{L}_{1}$ is a lower bidiagonal matrix with positive diagonal
entries. Notice that if $c$ is located in the right hand side of the
support, then you must deal with the Cholesky factorization of the matrix $c%
\mathbf{I}- \mathbf{J}.$

Next, we show that the five diagonal matrix $\mathbf{H}$ associated with (%
\ref{[Sec3]-Coeff-5TRR-2})\ can be given in terms of the five diagonal
matrix $(\mathbf{J}_{[2]}-c\mathbf{I)}^{2}$. Combining (\ref{FormMatrizH})
with (\ref{Prop-Sobpert2}), we get%
\begin{equation}
\mathbf{T}(x-c)^{2}\mathbf{\bar{p}}^{[2]}=\mathbf{HT\,\bar{p}}^{[2]}.
\label{Prop-03}
\end{equation}%
Substituting (\ref{Prop-04}) into (\ref{Prop-03})%
\begin{equation*}
\mathbf{T}\left( \mathbf{J}_{[2]}-c\mathbf{I}\right) ^{2}\mathbf{\,\bar{p}}%
^{[2]}=\mathbf{HT\,\bar{p}}^{[2]}.
\end{equation*}%
Hence, we state the following



\begin{proposition}
\label{PROP02}The semi-infinite five diagonal matrix $\mathbf{H}$, can be
obtained from the matrix $\left( \mathbf{J}_{[2]}-c\mathbf{I}\right) ^{2}$
as follows%
\begin{equation*}
\mathbf{H}=\mathbf{T}\left( \mathbf{J}_{[2]}-c\mathbf{I}\right) ^{2}\mathbf{T%
}^{-1}.
\end{equation*}
\end{proposition}



Next, we repeat the above process commuting the order of factors in $\mathbf{%
L}_{1}\mathbf{L}_{1}^{\intercal }$, Thus%
\begin{equation}
\mathbf{L}_{1}^{\intercal }\mathbf{L}_{1}=\mathbf{J}_{[2]}-c\mathbf{I}.
\label{segunda}
\end{equation}%
From (\ref{Chol01}) we have $\mathbf{L}_{1}=\mathbf{L}^{\intercal }\mathbf{LL%
}_{1}^{-\intercal }$, and replacing this expression as above, it yields%
\begin{equation*}
\mathbf{J}_{[2]}-cI=\mathbf{L}_{1}^{\intercal }\left( \mathbf{L}^{\intercal }%
\mathbf{LL}_{1}^{-\intercal }\right) =\left( \mathbf{L}_{1}^{\intercal }%
\mathbf{L}^{\intercal }\right) \left( \mathbf{LL}_{1}^{-\intercal }\right)
=\left( \mathbf{LL}_{1}\right) ^{\intercal }\left( \mathbf{LL}%
_{1}^{-\intercal }\right) =\mathbf{RQ}.
\end{equation*}%
Notice that $\mathbf{R}=\left( \mathbf{LL}_{1}\right) ^{\intercal }$ is an
upper triangular matrix, with positive diagonal entries because $\mathbf{L}$
and $\mathbf{L}_{1}$ are lower bidiagonal matrices. Now for the matrix $%
\mathbf{Q}=\mathbf{LL}_{1}^{-\intercal }$ we have%
\begin{eqnarray*}
\mathbf{QQ}^{\intercal } &=&\mathbf{LL}_{1}^{-\intercal }\left( \mathbf{LL}%
_{1}^{-\intercal }\right) ^{\intercal }=\mathbf{LL}_{1}^{-\intercal }\left(
\mathbf{L}_{1}^{-1}\mathbf{L}^{\intercal }\right) \\
&=&\mathbf{L}\left( \mathbf{L}_{1}^{-\intercal }\mathbf{L}_{1}^{-1}\right)
\mathbf{L}^{\intercal }=\mathbf{L}\left( \mathbf{L}_{1}\mathbf{L}%
_{1}^{\intercal }\right) ^{-1}\mathbf{L}^{\intercal }.
\end{eqnarray*}%
Next, from (\ref{Chol01}) $\mathbf{L}_{1}\mathbf{L}_{1}^{\intercal }=\mathbf{%
L}^{\intercal }\mathbf{L}.$ Thus
\begin{equation*}
\mathbf{QQ}^{\intercal }=\mathbf{L}\left( \mathbf{L}^{\intercal }\mathbf{L}%
\right) ^{-1}\mathbf{L}^{\intercal }=\mathbf{LL}^{-1}\mathbf{L}^{-\intercal }%
\mathbf{L}^{\intercal }=\mathbf{I},
\end{equation*}%
as well as
\begin{equation*}
\mathbf{Q}^{\intercal }\mathbf{Q}=\left( \mathbf{L}_{1}^{-1}\mathbf{L}%
^{\intercal }\right) \left( \mathbf{L}\mathbf{L}_{1}^{-\intercal }\right) =%
\mathbf{L}_{1}^{-1}\left( \mathbf{L}^{\intercal }\mathbf{L}\right) \mathbf{L}%
_{1}^{-\intercal }=\mathbf{L}_{1}^{-1}\left( \mathbf{L}_{1}\mathbf{L}%
_{1}^{\intercal }\right) \mathbf{L}_{1}^{-\intercal }=\mathbf{I}.
\end{equation*}%
This means that $\mathbf{Q}$ is an orthogonal matrix. Thus, we have proved
the following



\begin{proposition}
\label{PROP03}The positive definite matrix $\mathbf{J}_{[2]}-c\mathbf{I}$
can be factorised as follows
\begin{equation}
\mathbf{J}_{[2]}-c\mathbf{I}=\mathbf{RQ},  \label{ChoProp1}
\end{equation}%
where $\mathbf{R}$ is an upper triangular matrix, and\ $\mathbf{Q}$ is an
orthogonal matrix, i. e. $\mathbf{QQ}^{\intercal }=\mathbf{Q}^{\intercal }%
\mathbf{Q}=\mathbf{I}$.
\end{proposition}

Notice that the above result has been also proved in\cite{Ga02} but the fact
that also $\mathbf{QQ}^{\intercal }=\mathbf{I}$ holds is not proved.



Taking into account the previous result, we come back to (\ref{Chol00}) to
observe%
\begin{equation*}
\mathbf{J}-c\mathbf{I}=\mathbf{LL}^{\intercal }=\mathbf{L\left( \mathbf{L}%
_{1}^{-\intercal }\mathbf{L}_{1}^{\intercal }\right) L}^{\intercal }=\left(
\mathbf{L\mathbf{L}_{1}^{-\intercal }}\right) \left( \mathbf{\mathbf{L}%
_{1}^{\intercal }L}^{\intercal }\right) =\left( \mathbf{L\mathbf{L}%
_{1}^{-\intercal }}\right) \left( \mathbf{L\mathbf{L}_{1}}\right)
^{\intercal }=\mathbf{QR}.
\end{equation*}%
Thus, we can summarize the above as follows



\begin{proposition}
\label{PROP04}Let $\mathbf{J}$ be the symmetric Jacobi matrix such that
\begin{equation*}
x\,\mathbf{\bar{p}}=\mathbf{J}\,\mathbf{\bar{p}}.
\end{equation*}%
If $\mathbf{\bar{p}=}[p_{0}(x),p_{1}(x),p_{2}(x),\ldots ]^{\intercal }$ is
the infinite vector associated with the orthonormal polynomial sequence with
respect to $d\mu $ and we assume $p_{n}(c)\neq 0$ for $n\geq 1$, then the
following factorization%
\begin{equation*}
\mathbf{J}-c\mathbf{I}=\mathbf{QR}
\end{equation*}%
holds, Here $\mathbf{R}$ is an upper triangular matrix, and\ $\mathbf{Q}$ is
an orthogonal matrix , i.e. $\mathbf{QQ}^{\intercal }=\mathbf{Q}^{\intercal }%
\mathbf{Q}=\mathbf{I}$. Under these conditions,%
\begin{equation*}
\mathbf{RQ}=\mathbf{J}_{[2]}-c\mathbf{I},
\end{equation*}%
where $\mathbf{J}_{[2]}$ is the symmetric Jacobi matrix such that $x\,%
\mathbf{\bar{p}}^{[2]}=\mathbf{J}^{[2]}\,\mathbf{\bar{p}}^{[2]}$, where $%
\mathbf{\bar{p}}^{[2]}$\ is the infinite vector associated with the
orthonormal polynomial sequence with respect to $(x-c)^{2}d\mu $.
\end{proposition}



Observe that this is an alternative proof of Theorem 3.3 in \cite{BI-JCAM92}.

\smallskip

Since $\mathbf{J}_{[2]}$ is a symmetric matrix, from (\ref{ChoProp1}) and $%
\mathbf{QQ}^{\intercal }=\mathbf{I}$, we easily observe%
\begin{equation*}
\left( \mathbf{J}_{[2]}-c\mathbf{I}\right) ^{2}=\left( \mathbf{J}_{[2]}-c%
\mathbf{I}\right) \left( \mathbf{J}_{[2]}-c\mathbf{I}\right) ^{\intercal }=%
\mathbf{RQQ}^{\intercal }\mathbf{R}^{\intercal }=\mathbf{RR}^{\intercal }.
\end{equation*}%
Thus



\begin{proposition}
\label{PROP05}The square of the positive definite symmetric matrix $\mathbf{J%
}_{[2]}-c\mathbf{I}$ has the following factorization%
\begin{equation}
\left( \mathbf{J}_{[2]}-c\mathbf{I}\right) ^{2}=\mathbf{RR}^{\intercal },
\label{J2-cIeqRRt}
\end{equation}%
where $\mathbf{R}$ is an upper triangular matrix. Furthermore
\begin{equation*}
\left( \mathbf{J}-c\mathbf{I}\right) ^{2}=\mathbf{R}^{\intercal }\mathbf{R}.
\end{equation*}
\end{proposition}



Next, we are ready to prove that there is a very close relation between the
five diagonal semi-infinite symmetric matrix $\mathbf{H}$ defined in (\ref%
{MatrixH}), and the lower triangular, semi-infinite, nonsingular matrix $%
\mathbf{T}$ defined in (\ref{MatrixT}).

We will use the following notation. Let $\mathbf{\bar{f}}$ be any
semi-infinite column vector with polynomial entries $\mathbf{\bar{f}}%
=[f_{0}(x),f_{1}(x),f_{2}(x),\ldots ]^{\intercal }$. Then $\langle \mathbf{%
\bar{f}},\mathbf{\bar{g}}\rangle $ will represent the given inner product of
$\mathbf{\bar{f}}$ and $\mathbf{\bar{g}}$ componentwise, that is, we get the
following semi-infinite square matrix
\begin{equation*}
\langle \mathbf{\bar{f}},\mathbf{\bar{g}}\rangle =%
\begin{bmatrix}
\langle f_{0},g_{0}\rangle & \langle f_{0},g_{1}\rangle & \langle
f_{0},g_{2}\rangle & \cdots \\
\langle f_{1},g_{0}\rangle & \langle f_{1},g_{1}\rangle & \langle
f_{1},g_{2}\rangle & \cdots \\
\langle f_{2},g_{0}\rangle & \langle f_{2},g_{1}\rangle & \langle
f_{2},g_{2}\rangle & \cdots \\
\vdots & \vdots & \vdots & \ddots%
\end{bmatrix}%
.
\end{equation*}%
Next, let us recall (\ref{Prop-Sobpert2}), i. e. $\mathbf{\bar{s}}^{M,N}=%
\mathbf{T\,\bar{p}}^{[2]}.$ Let us consider the inner product%
\begin{equation*}
\langle \mathbf{\bar{s}}^{M,N},\mathbf{\bar{s}}^{M,N}\rangle _{\lbrack
2]}=\langle \mathbf{T\,\bar{p}}^{[2]},\mathbf{T\,\bar{p}}^{[2]}\rangle
_{\lbrack 2]}=\mathbf{T}\langle \mathbf{\bar{p}}^{[2]},\mathbf{\bar{p}}%
^{[2]}\rangle _{\lbrack 2]}\mathbf{T}^{\intercal }=\mathbf{T}\mathbf{T}%
^{\intercal },
\end{equation*}%
where $\langle \mathbf{\bar{p}}^{[2]},\mathbf{\bar{p}}^{[2]}\rangle
_{\lbrack 2]}=\mathbf{I}$ because we deal with orthonormal polynomials. On
the other hand, from (\ref{S1-DicrSob-Inn}) and (\ref{FormMatrizH}), one has%
\begin{equation*}
\langle \mathbf{\bar{s}}^{M,N},\mathbf{\bar{s}}^{M,N}\rangle _{\lbrack
2]}=\langle (x-c)^{2}\mathbf{\bar{s}}^{M,N},\mathbf{\bar{s}}^{M,N}\rangle
_{S}=\langle \mathbf{H\,\bar{s}}^{M,N},\mathbf{\bar{s}}^{M,N}\rangle _{S}=%
\mathbf{H\,}\langle \mathbf{\bar{s}}^{M,N},\mathbf{\bar{s}}^{M,N}\rangle
_{S}=\mathbf{H},
\end{equation*}%
where again $\langle \mathbf{\bar{s}}^{M,N},\mathbf{\bar{s}}^{M,N}\rangle
_{S}=\mathbf{I}$ since we deal with orthonormal polynomials. Thus, we have
proved the following



\begin{proposition}
\label{PROP06}The five diagonal semi-infinite symmetric matrix $\mathbf{H}$
defined in (\ref{MatrixH}), has the following Cholesky factorization%
\begin{equation*}
\mathbf{H}=\mathbf{TT}^{\intercal },
\end{equation*}%
where $\mathbf{T}$ is the lower triangular, semi-infinite matrix defined in (%
\ref{MatrixT}).
\end{proposition}



Finally, from (\ref{Prop-Sobpert2}) and (\ref{FormMatrizH}), we have%
\begin{equation*}
(x-c)^{2}\mathbf{\bar{s}}^{M,N}=\mathbf{H\,\bar{s}}^{M,N}=(x-c)^{2}\mathbf{%
T\,\bar{p}}^{[2]}=\mathbf{T}(x-c)^{2}\mathbf{\bar{p}}^{[2]}.
\end{equation*}%
According to (\ref{Prop-04}), we get%
\begin{equation*}
\mathbf{T}(x-c)^{2}\mathbf{\bar{p}}^{[2]}=\mathbf{T}\left( \mathbf{J}_{[2]}-c%
\mathbf{I}\right) ^{2}\mathbf{\bar{p}}^{[2]}.
\end{equation*}%
Next, from (\ref{J2-cIeqRRt}) we obtain
\begin{equation*}
\mathbf{T}\left( \mathbf{J}_{[2]}-c\mathbf{I}\right) ^{2}\mathbf{\bar{p}}%
^{[2]}=\mathbf{T\mathbf{RR}^{\intercal }\bar{p}}^{[2]}=\mathbf{H\,\bar{s}}%
^{M,N}=\mathbf{TT}^{\intercal }\mathbf{T\,\bar{p}}^{[2]}.
\end{equation*}%
Therefore,%
\begin{equation*}
\mathbf{T\mathbf{RR}^{\intercal }\bar{p}}^{[2]}=\mathbf{TT}^{\intercal }%
\mathbf{T\,\bar{p}}^{[2]}
\end{equation*}%
and, as a consequence,%
\begin{equation*}
\mathbf{\mathbf{RR}^{\intercal }}=\mathbf{T}^{\intercal }\mathbf{T}.
\end{equation*}



\begin{proposition}
\label{PROP07}For any positive Borel measure $d\mu $ supported on $%
E\subseteq \mathbb{R}$, if $\mathbf{J}$ is the corresponding semi-infinite
symmetric Jacobi matrix and $c\notin E,$ then for the $2-$iterated perturbed
measure $(x-c)^{2}d\mu $ such that if $\mathbf{J}_{[2]}$\ is the
corresponding semi-infinite symmetric Jacobi matrix we get%
\begin{equation*}
\left( \mathbf{J}_{[2]}-c\mathbf{I}\right) ^{2}=\mathbf{RR}^{\intercal }=%
\mathbf{T}^{\intercal }\mathbf{T}.
\end{equation*}
\end{proposition}

This is the symmetric version of Theorem 5.3 in \cite{DM-NA14} , where the
authors use other kind of factorization based on monic orthogonal
polynomials.



\section{An example with Laguerre polynomials}

\label{Sec7-ExampleLaguerre}



In \cite{HPMQ-subm13} and Section \ref{Sec5-5TRR}\ the coefficients of (\ref%
{[Sec3]-Coeff-5TRR-2}) for the monic Laguerre Sobolev-type orthogonal
polynomials have been deduced. In the sequel we illustrate the matrix
approach presented in the previous section for the Laguerre case with a
particular example. First, let us denote by $\{\ell _{n}^{\alpha
}(x)\}_{n\geq 0}$, $\{\ell _{n}^{\alpha ,[2]}(x)\}_{n\geq 0}$, $%
\{s_{n}^{M,N}(x)\}_{n\geq 0}$ the sequences of orthonormal polynomials
polynomials with respect to the inner products (\ref{S1-InnProd-mu}), (\ref%
{S1-k-iter-OP}) and (\ref{S1-DicrSob-Inn}), respectively, when $d\mu
(x)=x^{\alpha }e^{-x}dx,\alpha >-1,$ is the classical Laguerre weight
function supported on $(0,+\infty )$.

\smallskip

In order to obtain compact expressions of the matrices, in this section we
will particularize all of those presented in the previous section for the
choice of the parameters $\alpha =0$, $c=-1$, $M=1$, and $N=1$. In these
conditions, using any symbolic algebra package as, for example, Wolfram
Mathematica\copyright, the explicit expressions of the sequences of orthogonal
polynomials appearing in our study.

From section \ref{Sec5-5TRR} we know%
\begin{equation*}
\mathbf{H}=%
\begin{bmatrix}
\frac{5}{2} & \frac{11}{2\sqrt{2}} & \frac{1}{2}\sqrt{\frac{89}{2}} &  &  &
\\
\frac{11}{2\sqrt{2}} & \frac{19}{2} & \frac{129}{\sqrt{89}} & \frac{1}{2}%
\sqrt{\frac{35705}{178}} &  &  \\
\frac{1}{2}\sqrt{\frac{89}{2}} & \frac{129}{\sqrt{89}} & \frac{5331}{178} &
\frac{1503493}{178\sqrt{71410}} & 4\sqrt{\frac{26690173}{3177745}} &  \\
& \frac{1}{2}\sqrt{\frac{35705}{178}} & \frac{1503493}{178\sqrt{71410}} &
\frac{415128273}{6355490} & \frac{72140663342}{35705}\sqrt{\frac{2}{%
2375425397}} & \ddots \\
&  & 4\sqrt{\frac{26690173}{3177745}} & \frac{72140663342}{35705}\sqrt{\frac{%
2}{2375425397}} & \frac{108116532681297}{952972626965} & \ddots \\
&  &  & \ddots & \ddots & \ddots%
\end{bmatrix}%
\end{equation*}%
On the other hand, from (\ref{Prop-Sobpert2}) we obtain%
\begin{equation}
\mathbf{T}=%
\begin{bmatrix}
\sqrt{\frac{5}{2}} &  &  &  &  &  \\
\frac{11}{2\sqrt{5}} & \frac{1}{2}\sqrt{\frac{69}{5}} &  &  &  &  \\
\frac{1}{2}\sqrt{\frac{89}{5}} & \frac{1601}{2\sqrt{30705}} & 4\sqrt{\frac{%
1777}{6141}} &  &  &  \\
& 5\sqrt{\frac{7141}{12282}} & \frac{2911082\sqrt{2}}{\sqrt{389632847685}} &
6\sqrt{\frac{346922}{1714805}} &  &  \\
&  & \sqrt{\frac{1841621937}{63447785}} & \frac{2555758506}{\sqrt{%
89202693674855485}} & 12\sqrt{\frac{4046188065}{52019147177}} &  \\
&  &  & \ddots & \ddots & \ddots%
\end{bmatrix}%
.  \label{matrizT}
\end{equation}%
Notice that if we multiply $\mathbf{T}$ by its transpose then one recovers $%
\mathbf{H}$ according to the statement of Proposition \ref{PROP06}.

The tridiagonal symmetric Jacobi matrix associated with the standard
orthonormal family $\{\ell _{n}^{\alpha ,[2]}(x)\}_{n\geq 0}$ reads%
\begin{equation*}
\mathbf{J}_{[2]}=%
\begin{bmatrix}
\frac{11}{5} & \frac{\sqrt{69}}{5} &  &  &  &  \\
\frac{\sqrt{69}}{5} & \frac{1501}{345} & \frac{2\sqrt{8885}}{69} &  &  &  \\
& \frac{2\sqrt{8885}}{69} & \frac{790903}{122613} & \frac{3\sqrt{4975797}}{%
1777} &  &  \\
&  & \frac{3\sqrt{4975797}}{1777} & \frac{1091564609}{128144801} & \frac{4%
\sqrt{7450856157}}{72113} &  \\
&  &  & \frac{4\sqrt{7450856157}}{72113} & \frac{3195035811691}{302365554333}
& \ddots \\
&  &  &  & \ddots & \ddots%
\end{bmatrix}%
,
\end{equation*}%
and from this expression it is straightforward to check Proposition \ref%
{PROP02}. Next, from the symmetric Jacobi matrix

\begin{equation}
\mathbf{J}=%
\begin{bmatrix}
1 & 1 &  &  &  &  \\
1 & 3 & 2 &  &  &  \\
& 2 & 5 & 3 &  &  \\
&  & 3 & 7 & 4 &  \\
&  &  & 4 & 9 & \ddots \\
&  &  &  & \ddots & \ddots%
\end{bmatrix}
\label{matrizJ}
\end{equation}%
associated with $\{\ell _{n}^{\alpha }(x)\}_{n\geq 0}$, we can implement the
Cholesky factorization of $\mathbf{J-}c\mathbf{I=\mathbf{LL}^{\intercal }}$
in such a way the lower bidiagonal matrix is
\begin{equation*}
\mathbf{L}=%
\begin{bmatrix}
\sqrt{2} &  &  &  &  &  \\
\frac{1}{\sqrt{2}} & \sqrt{\frac{7}{2}} &  &  &  &  \\
& 2\sqrt{\frac{2}{7}} & \sqrt{\frac{34}{7}} &  &  &  \\
&  & 3\sqrt{\frac{7}{34}} & \sqrt{\frac{209}{34}} &  &  \\
&  &  & 4\sqrt{\frac{34}{209}} & \sqrt{\frac{1546}{209}} &  \\
&  &  &  & \ddots & \ddots%
\end{bmatrix}%
.
\end{equation*}%
Following (\ref{Chol01}), we commute the order of $\mathbf{\mathbf{L}}$ and
its transpose to obtain $\mathbf{\mathbf{L}^{\intercal }\mathbf{L}=J}_{[1]}-c%
\mathbf{I}$, where%
\begin{equation*}
\mathbf{J}_{[1]}=%
\begin{bmatrix}
\frac{3}{2} & \frac{\sqrt{7}}{2} &  &  &  &  \\
\frac{\sqrt{7}}{2} & \frac{51}{14} & \frac{4\sqrt{17}}{7} &  &  &  \\
& \frac{4\sqrt{17}}{7} & \frac{1359}{238} & \frac{3\sqrt{1463}}{34} &  &  \\
&  & \frac{3\sqrt{1463}}{34} & \frac{55071}{7106} & \frac{8\sqrt{13141}}{209}
&  \\
&  &  & \frac{8\sqrt{13141}}{209} & \frac{3159027}{323114} & \ddots \\
&  &  &  & \ddots & \ddots%
\end{bmatrix}%
.
\end{equation*}%
The computation of a new Cholesky factorization of $\mathbf{J}_{[1]}-c%
\mathbf{I}$ yields $\mathbf{\mathbf{L}_{1}\mathbf{L}_{1}^{\intercal }}$,
where%
\begin{equation*}
\mathbf{L}_{1}=%
\begin{bmatrix}
\sqrt{\frac{5}{2}} &  &  &  &  &  \\
\sqrt{\frac{7}{10}} & \sqrt{\frac{138}{35}} &  &  &  &  \\
& 2\sqrt{\frac{170}{483}} & \sqrt{\frac{12439}{2346}} &  &  &  \\
&  & 3\sqrt{\frac{14421}{60418}} & \sqrt{\frac{2451842}{371393}} &  &  \\
&  &  & 4\sqrt{\frac{2747242}{15071617}} & \sqrt{\frac{876324669}{111486698}}
&  \\
&  &  &  & \ddots & \ddots%
\end{bmatrix}%
.
\end{equation*}%
Commuting the order of the matrices in the decomposition then we finally
deduce the expression (\ref{segunda}), i. e. $\mathbf{L}_{1}^{\intercal }%
\mathbf{L}_{1}=\mathbf{J}_{[2]}-c\mathbf{I}$. With these last matrices in
mind we find $\mathbf{R}$ and $\mathbf{Q}$ at Proposition \ref{PROP03}. Thus%
\begin{equation}
\mathbf{Q}=\mathbf{LL}_{1}^{-\intercal }=%
\begin{bmatrix}
\frac{2}{\sqrt{5}} & \frac{-7}{\sqrt{345}} & \frac{68}{\sqrt{122613}} &
\frac{-1254}{\sqrt{128144801}} & \frac{12368\sqrt{3}}{\sqrt{100788518111}} &
\cdots \\
\frac{1}{\sqrt{5}} & \frac{14}{\sqrt{345}} & \frac{-136}{\sqrt{122613}} &
\frac{2508}{\sqrt{128144801}} & \frac{-24736\sqrt{3}}{\sqrt{100788518111}} &
\cdots \\
& 2\sqrt{\frac{5}{69}} & \frac{-238}{\sqrt{122613}} & \frac{-4389}{\sqrt{%
128144801}} & \frac{43288\sqrt{3}}{\sqrt{100788518111}} & \cdots \\
&  & 3\sqrt{\frac{69}{1777}} & \frac{7106}{\sqrt{128144801}} & \frac{-210256%
}{\sqrt{302365554333}} & \cdots \\
&  &  & 4\sqrt{\frac{1777}{72113}} & \frac{323114}{\sqrt{302365554333}} &
\cdots \\
&  &  &  & \ddots & \ddots%
\end{bmatrix}
\label{matrizQ}
\end{equation}%
and%
\begin{equation}
\mathbf{R}=\left( \mathbf{LL}_{1}\right) ^{\intercal }=%
\begin{bmatrix}
\sqrt{5} & \frac{6}{\sqrt{5}} & \frac{2}{\sqrt{5}} &  &  &  \\
& \sqrt{\frac{69}{5}} & \frac{88}{\sqrt{345}} & 2\sqrt{\frac{15}{23}} &  &
\\
&  & \sqrt{\frac{1777}{69}} & 790\sqrt{\frac{3}{40871}} & 12\sqrt{\frac{69}{%
1777}} &  \\
&  &  & \sqrt{\frac{72113}{1777}} & \frac{99504}{\sqrt{128144801}} & \ddots
\\
&  &  &  & \sqrt{\frac{4192941}{72113}} & \ddots \\
&  &  &  &  & \ddots%
\end{bmatrix}%
.  \label{matrizR}
\end{equation}%
Observe that $\mathbf{Q}$ is a matrix whose rows are orthogonal vectors, and
multiplying (\ref{matrizQ}) above by its transpose (in this order) we get%
\begin{equation*}
\mathbf{QQ}^{\intercal }\approx
\begin{bmatrix}
0.99657 & 0.0068687 & -0.27601 & 0.019461 & -0.029907 & \cdots \\
0.0068687 & 0.98626 & 0.55201 & -0.038923 & 0.059815 & \cdots \\
-0.27601 & 0.55201 & 0.95793 & -0.73549 & -0.10468 & \cdots \\
0.019461 & -0.038923 & -0.73549 & 0.88972 & 0.16948 & \cdots \\
-0.029907 & 0.059815 & -0.10468 & 0.16948 & 0.73956 & \cdots \\
\vdots & \vdots & \vdots & \vdots & \vdots & \ddots%
\end{bmatrix}%
\approx \mathbf{I}.
\end{equation*}%
Notice that we implement our algorithm with finite matrices. Notwithstanding
the foregoing, multiplying the transpose of (\ref{matrizQ}) by (\ref{matrizQ}%
) we indeed have $\mathbf{Q}^{\intercal }\mathbf{Q}=\mathbf{I}$.

Employing these matrices above it is easy to test numerically expressions $%
\mathbf{H}=\mathbf{T}\left( \mathbf{J}_{[2]}-c\mathbf{I}\right) ^{2}\mathbf{T%
}^{-1}$, $\mathbf{J}-c\mathbf{I}=\mathbf{QR}$, and $\mathbf{J}_{[2]}-c%
\mathbf{I}=\mathbf{RQ}$ according to the statements of Propositions \ref%
{PROP02}, \ref{PROP03} \ref{PROP04} respectively. It is also possible to
check that using the numerical expression (\ref{matrizT}), and alternatively
the expression (\ref{matrizR}), we recover%
\begin{equation*}
\left( \mathbf{J}_{[2]}-c\mathbf{I}\right) ^{2}=
\end{equation*}%
\begin{equation}
\begin{bmatrix}
13 & \frac{118}{\sqrt{69}} & 2\sqrt{\frac{1777}{345}} &  &  &  \\
\frac{118}{\sqrt{69}} & \frac{2681}{69} & \frac{227476}{69\sqrt{8885}} & 2%
\sqrt{\frac{1081695}{40871}} &  &  \\
2\sqrt{\frac{1777}{345}} & \frac{227476}{69\sqrt{8885}} & \frac{9460213}{%
122613} & \frac{84432374\sqrt{\frac{3}{1658599}}}{1777} & 36\sqrt{\frac{%
32145881}{128144801}} &  \\
& 2\sqrt{\frac{1081695}{40871}} & \frac{84432374\sqrt{\frac{3}{1658599}}}{%
1777} & \frac{16364422385}{128144801} & \frac{628405520264}{72113\sqrt{%
7450856157}} & \ddots  \\
&  & 36\sqrt{\frac{32145881}{128144801}} & \frac{628405520264}{72113\sqrt{%
7450856157}} & \frac{57572534044081}{302365554333} & \ddots  \\
&  &  & \ddots  & \ddots  & \ddots
\end{bmatrix}%
,  \label{matrizJ2sq}
\end{equation}%
according to Proposition \ref{PROP07}.

Finally, Proposition \ref{PROP05} can be numerically tested from (\ref%
{matrizJ}), (\ref{matrizR}) and (\ref{matrizJ2sq}).



\section*{Acknowledgments}



The work of CH, EJH and AL is supported by Direcci\'{o}n General de Investigaci\'{o}%
n e Innovaci\'{o}n, Consejer\'{i}a de Educaci\'{o}n e Investigaci\'{o}n of the Comunidad de
Madrid (Spain), and Universidad de Alcal\'{a} under grants CM/JIN/2019-010 and
CM/JIN/2021-014, \textit{Proyectos de I+D para J\'{o}venes Investigadores de la
Universidad de Alcal\'{a} 2019 and 2021, respectively}. The work of FM has been
supported by FEDER/Ministerio de Ciencia e Innovaci\'{o}n-Agencia Estatal de
Investigaci\'{o}n of Spain, grant PGC2018-096504-B-C33, and the Madrid
Government (Comunidad de Madrid-Spain) under the Multiannual Agreement with
UC3M in the line of Excellence of University Professors, grant EPUC3M23 in
the context of the V PRICIT (Regional Programme of Research and
Technological Innovation).




\end{document}